\documentclass[11pt]{article}
\pdfoutput=1

\usepackage[T1]{fontenc}
\usepackage[utf8]{inputenc}
\usepackage{lmodern}
\usepackage{microtype}

\usepackage{amsmath,amsthm,amsfonts,amssymb,amscd,mathrsfs}
\usepackage{graphicx}
\usepackage[table]{xcolor}
\usepackage{subcaption}
\usepackage{booktabs,multirow}
\usepackage{enumitem}
\usepackage{float}
\usepackage{setspace}
\usepackage{xurl}
\usepackage[numbers]{natbib}
\usepackage[colorlinks=true,linkcolor=black,citecolor=black,urlcolor=black]{hyperref}
\usepackage[a4paper,margin=2.5cm]{geometry}

\numberwithin{equation}{section}

\newtheorem{theorem}{Theorem}[section]

\newtheorem{corollary}[theorem]{Corollary}
\newtheorem{definition}[theorem]{Definition}
\newtheorem{lemma}[theorem]{Lemma}
\newtheorem{proposition}[theorem]{Proposition}

\theoremstyle{remark}
\newtheorem{remark}[theorem]{Remark}
\newtheorem{example}[theorem]{Example}

\newcommand{\F}{\mathcal F}
\newcommand{\W}{\mathcal W}
\newcommand{\N}{\mathbb N}
\newcommand{\R}{\mathbb R}

\newcommand{\B}{\mathcal B}
\newcommand{\Pc}{\mathcal P}
\newcommand{\supp}{\text{supp}}
\newcommand{\law}{\text{Law}}
\def\P{{\mathbb P}}
\def\E{{\mathbb E}}

\newcommand{\tm}[1]{\| #1 \|}

\def\Q{{\mathbb Q}}

\newcommand{\suc}{\prec_{\text{sc}}}
\newcommand{\bbc}{\prec_{\text{$\beta$}}}
\newcommand{\sbbc}{\prec_{\text{s$\beta$}}}

\newcommand{\cplbbc}{{\text{Cpl}}_{\text{$\beta$}}}
\newcommand{\cplsbbc}{{\text{Cpl}}_{\text{s$\beta$}}}
\newcommand{\cpls}{{\text{Cpl}}}

\begin{document}

\title{Strassen's theorem for biased convex order}

\author{
Beatrice Acciaio\thanks{Department of Mathematics, ETH Zürich, Switzerland. \texttt{beatrice.acciaio@math.ethz.ch}}
\and
Mathias Beiglböck\thanks{Faculty of Mathematics, University of Vienna, Austria. \texttt{mathias.beiglboeck@univie.ac.at}}
\and
Evgeny Kolosov\thanks{Department of Mathematics, ETH Zürich, Switzerland. \texttt{evgeny.kolosov@math.ethz.ch}}
\and
Gudmund Pammer\thanks{Institute of Statistics, TU Graz, Austria. \texttt{gudmund.pammer@tugraz.at}}
}

\maketitle

\begin{abstract}
Strassen's theorem asserts that for given marginal probabilities $\mu, \nu$ there exists a martingale starting in $\mu$ and terminating in $\nu $ if and only if $\mu, \nu$ are in convex order. From a financial perspective, it guarantees the existence of market consistent martingale pricing measures for arbitrage-free prices of European call options and thus plays a fundamental role in robust finance. 
Arbitrage-free prices of \emph{American} options demand for a stronger version of martingales which are `biased' in a specific sense. In this paper, we derive an extension of Strassen's theorem which links them to an appropriate strengthening of the convex order. Moreover, we provide a characterization of this order through integrals with respect to compensated Poisson processes. 
\end{abstract}

\section{Introduction}
While classical results characterize martingales with prescribed marginals via convex order and are sufficient to deal with no arbitrage conditions for European options, more nuanced notions are required to account for the valuation of American options. 
To explain this rigorously we need to set up some notations. Let $\Pc_1(\R)$ be the family of probability measures on the real line with finite first moment.
Recall that $\mu, \nu\in \Pc_1(\R)$  are in convex order, denoted $\mu\prec \nu$, if and only if $\int  \phi\, d\mu \leq \int \phi\, d\nu$ for every convex function $\phi$ on $\R$.
Strassen's theorem \cite{St65} asserts that for given marginal probabilities $\mu, \nu$ there exists a martingale starting in $\mu$ and terminating in $\nu $ if and only if $\mu\prec \nu$.
In mathematical finance, the convex order appears naturally from \emph{no arbitrage} considerations. Indeed, via the celebrated Breeden--Litzenberger observation \cite{BrLi78}, the prices of European call/put options for a given maturity can be encoded as the one-dimensional (marginal) distribution of the asset price process at that time. The no-arbitrage condition guarantees that these are in convex order and thus, by Strassen's theorem, there exists a martingale with these marginals. Hence Strassen's theorem can be regarded as a basic fundamental theorem of asset pricing for robust finance given European option prices, see e.g.\ \cite{HoNe12, AcBePeSc13, WiZh22} for details.

In contrast, as we describe in the forthcoming paper \cite{AcBeKoPa24}, arbitrage-free prices of American options can be described through a stronger notion of martingale which we introduce next. Let $\beta\in [0,1)$ and $(X_t)_{t=0,1}$ be a martingale on some stochastic basis $(\Omega, \F, \P, (\F_t)_{t=0,1})$. We say that  $X$ is  \emph{$\beta$-biased} if $\law(X_1|\F_0)$ has an atom of mass at least $\beta$ at the maximum of its support, i.e.\ $\P$-a.s.\ $\rho=\law(X_1|\F_0)$ satisfies\footnote{We make the convention that $\{\max S\} = \emptyset$ if $S$ is not upper bounded such that $\rho(\{\max S\})=0$ in this case.}
\begin{align*} 
    \rho(\{ \max(\supp (\rho))\})\geq \beta.
\end{align*}
Observe that this condition gets stricter as $\beta$ increases. While $0$-biased martingales are just ordinary martingales, $1$-biased martingales have to be constant. 
We call the law of a $\beta$-biased martingale a \emph{$\beta$-biased martingale coupling} and
 denote by $\cplbbc(\mu, \nu)$ the set of all such couplings which have marginals $\mu, \nu$. 

Following the role model of Strassen's theorem, our first goal is to characterize when $\cplbbc(\mu, \nu)$ is non-empty. 
To this end, we define for any continuous, linearly bounded $g:\R\to\R$ its $\beta$-envelope
\begin{equation}\label{eq:beta_transform}
g_\beta(x):=\textstyle \inf\{\int g \, d\rho:\rho \in \Pc_1(\R), \int y \, d\rho= x, \rho({\max \supp (\rho)})\geq \beta\}.
\end{equation}
 We say that $\mu, \nu\in \Pc_1(\R)$ are in \emph{$\beta$-biased order}, denoted  $\mu\bbc\nu$, if   $\mu(g_\beta) \leq \nu(g)$ for all continuous linearly bounded $g:\R\to \R$.
Note that for $\beta=0$,  $g_0$ is simply the convex hull of $g$ and $\mu\prec_0\nu$ is simply the convex order. As $\beta$ increases, the order gets progressively stronger. In the limiting case $\beta=1$, $\mu, \nu$ are ordered only in the trivial case $\mu=\nu$. 

Our first main result is the following refinement of Strassen's theorem. 
\begin{theorem}[$\beta$-biased Strassen]
\label{thm:bbc-Strassen_intro}
Let $\mu, \nu \in \Pc_1(\R)$, $\beta \in [0,1)$. Then $\cplbbc(\mu, \nu)\neq \emptyset$ if and only if $\mu\bbc \nu$.
\end{theorem}

The classical convex order is also a necessary and sufficient conditions for two probabilities to appear as initial and terminal distribution in the Skorokhod embedding problem (see e.g.\ \cite{Sk61, Ho11, Ob04, BeCoHu14}) or  as initial and terminal marginal of Bass martingales which are particular integrals w.r.t.\ Brownian motion (see e.g.\ \cite{GuLoWa19, BaBeHuKa20, CoHe21, AcMaPa23}).
Notably, the $\beta$-biased order can also be characterized in terms of integrals, though with respect to compensated Poisson processes. To this end, we fix a standard Poisson process $(N_t)_{t\geq 0}$ and assume that we are working on a stochastic basis which is rich enough to support a continuous random variable at time $0$. 
For $t \geq 0$, let $M_t := t - N_t$ denote the (negative) compensated Poisson process. 
Then we have the following characterization.
\begin{theorem}\label{thm:order_integral_Representation_Intro}
    Let $\mu,\nu\in \Pc_1(\R), \beta \in (0,1)$. Then $\cplbbc(\mu, \nu)\neq \emptyset$  if and only if,  for some $X_0 \sim \mu$
    and predictable $H \geq 0$ with  
    $\E\big[ \int_0^{\infty} H_s \, ds \big] < \infty $,
    \begin{equation}\label{eq:thm_order^Representation}
         X_0 + \int_0^{\log(1/\beta)} H_s \,dM_s \sim \nu.
    \end{equation}
\end{theorem}
Note that Theorem \ref{thm:order_integral_Representation_Intro} also holds in the case $\beta=0$ if we interpret $\log(1/ \beta)$ as $+\infty$.
That is, if $\mu, \nu$ are in convex order, then for some $X_0 \sim \mu$    and predictable, integrable  $H \geq 0$ we have $
         X_0 + \int_0^{\infty} H_s \,dM_s \sim \nu. $

Simple examples show that $\prec_{\beta}$ is not transitive, i.e.\ $\mu\prec_{\beta}\nu$ and $\nu\prec_{\beta} \rho$ do not imply that $\mu \prec_{\beta} \rho$. However, by Theorem \ref{thm:order_integral_Representation_Intro} we have that $\mu\prec_{\beta_1}\nu$ and $\nu\prec_{\beta_2} \rho$ imply that $\mu\prec_{\beta_1 \beta_2} \rho$, and it is easy to see that this result cannot be improved. That is, one cannot find a constant $\tilde \beta  > \beta_1 \beta_2$ such that the results holds for $\prec_{\tilde \beta}$. 

Finally, in our financial applications \cite{AcBeKoPa24} we are also interested in martingales $X$ which are  
\emph{strongly $\beta$-biased} in the sense that $\law(X_1|\F_0)$ has an atom of mass strictly larger than $\beta$ at the maximum of its support, i.e.\ $\P$-a.s.\ $\rho=\law(X_1|\F_0)$ satisfies
\begin{align*}
    \rho(\{ \max \supp(\rho) \})> \beta.
\end{align*}
We call the law of a strongly  $\beta$-biased martingale a \emph{strongly $\beta$-biased martingale} coupling and
 denote by $\cplsbbc(\mu, \nu)$ the set of all such couplings which have marginals $\mu, \nu$. 
We will say that $\mu, \nu$ are in strong $\beta$-biased order, denoted by $\mu\sbbc\nu$, if $\cplsbbc(\mu, \nu)\neq \emptyset$.  In contrast to the above, we do not know a simple Strassen-type criterion for the strong order.
 However, we show the following characterization:  
\begin{theorem}\label{thm:strict_order_integral_Representation_Intro}
    Let $\mu,\nu\in \Pc_1(\R), \beta \in (0,1)$. Then  $\cplsbbc(\mu, \nu)\neq \emptyset$ if and only if,  for some $X_0 \sim \mu$,
predictable $H \geq 0$ with  
    $\E\big[ \int_0^{\infty} H_s \, ds \big] < \infty $, and stopping time $\tau < \log(1/\beta)$,
    \begin{equation}\label{eq:strict_thm_order^Representation}
         X_0 + \int_0^{\tau} H_s \,dM_s \sim \nu.
    \end{equation}
\end{theorem}

\paragraph{Motivation from  arbitrage-free pricing of American options.}
To give an intuition about the need of an order stronger than convex order, we consider the simple case of a one-step market model with a deterministic num\'eraire $B_0=1, B_1>1$ and a risky asset with value $s_0$ at time $0$. We  let $\beta = 1-1/B_1$  and assume that, for every $k\geq 0$, American put options on $S\geq0$ with maturity $1$ and strike $k$ are liquidly traded, say at price $p(k)$. Then one can prove (see \cite{AcBeKoPa24}) that those prices are compatible with no-arbitrage, in the sense that there exists a martingale $(S_t)_{t=0,1}$ on some stochastic basis $(\Omega, \F, \P, (\F_t)_{t=0,1})$  such that 
\[
p(k)=\sup_{\tau \in \{0,1\} \text{ stopping time}} \E_\Q[(k/B_\tau - S_\tau/B_\tau)_+ ],   
\quad k\geq 0,
\]
if and only if $p$ is the potential measure of a measure $\nu\in\mathcal{P}_1([0,+\infty))$ 
(in the sense $p(k)=\int (k-x)_+\nu(dx),\; k \geq 0$)  
 satisfying $$ \delta_{s_0}\prec_{s\beta} \nu. $$ 
Figure \ref{fig:crash_point_example} depicts $p$ in the case of trivial $\F_0$.
\begin{figure}[ht]
    \centering
    \includegraphics[width=0.5\textwidth]{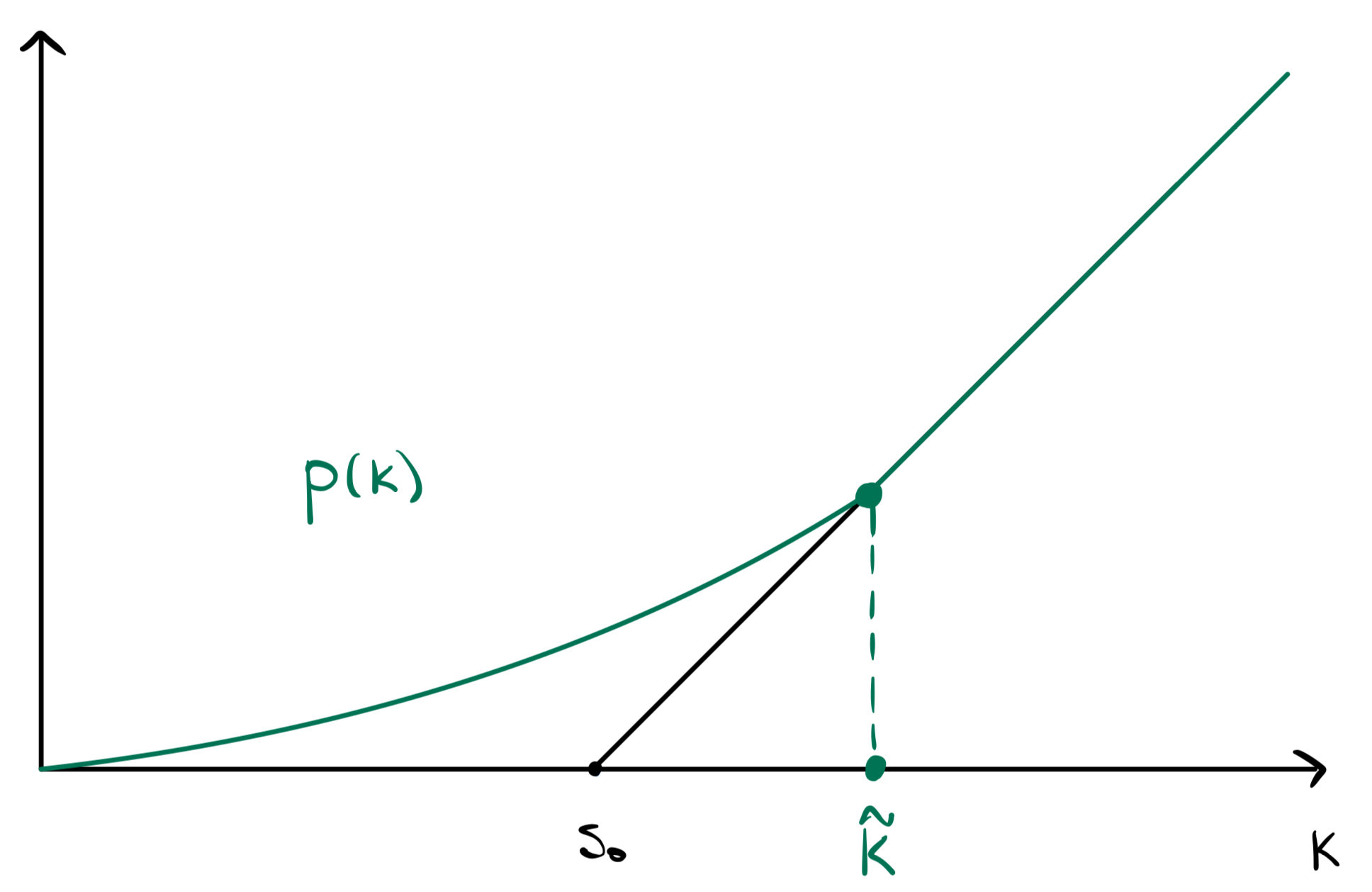}
    \caption{Price curve of American put options in a one-step model. For strikes $k \geq \tilde{k} = \max \supp(\nu)$, it is optimal to exercise the option at time $t = 0$; otherwise, at time $t = 1$.}
    \label{fig:crash_point_example}
\end{figure}

\paragraph{Notation.} 
We introduce here some notation used throughout the paper.  By $\Pc(\R)$ we denote the set of probability measures on $\R$, and by $\Pc_1(\R)$ the subset of measures with finite first moment.

For a measure $\mu$ on $\R$ with finite first moment, we write $\tm{\mu}:=\mu(\R)$ for its total mass, $\bar{\mu} := \frac{1}{\tm{\mu}} \int y \, \mu(dy)$ for its barycenter, and define 
\[
\mu^L := \mu|_{(-\infty,\bar{\mu})}\quad \text{and}\quad \mu^R := \mu|_{[\bar{\mu},+\infty)}.
\]
For $\mu,\nu \in \Pc(\R)$, we denote by $\cpls(\mu,\nu)$ the set of probability measures $\pi \in \Pc(\R\times\R)$ with first marginal $\mu$ and second marginal $\nu$, called couplings of $\mu$ and $\nu$. If $\mu,\nu \in \Pc_1(\R)$, a martingale coupling of $\mu$ and $\nu$ is a coupling $\pi \in \cpls(\mu,\nu)$ such that $\bar{\pi}_x = x$ $\mu$-a.s., where $(\pi_x)_x$ is the $\mu$-disintegration of $\pi$.
For $\mu,\nu \in \Pc_1(\R)$, we consider the 1-Wasserstein distance
\[
    \mathcal W_1(\mu,\nu) := \inf_{\pi \in \cpls(\mu,\nu)} \int |x-y| \, d\pi(x,y),
\]
and endow $\Pc_1(\R)$ with the topology induced by $\mathcal W_1$.
For finite measures $\mu$ and $\nu$ on $\R$ with finite first moment, we say that $\mu$ is smaller than $\nu$ in convex order, denoted $\mu \prec \nu$, if  
$\int \phi \, d\mu \leq \int \phi \, d\nu$  
for every convex function $\phi$ on $\R$.  
Given a measurable function $T: \R \to \R$ and a measure $\mu$ on $\R$, we write $T_{\#} \mu$ for the push-forward measure of $\mu$ under $T$, that is, $T_{\#} \mu(A) = \mu(T^{-1}(A))$ for any Borel set $A\in \B(\R)$.
For a measure $\nu$ on $\R$, we denote the extreme points of its support by
\begin{equation*} \label{eq:maxsupp_notation}
    s(\nu) := \inf \supp(\nu)\quad \text{and}\quad S(\nu) := \sup \supp(\nu).
\end{equation*}
We denote by $C_L(\R)$ the space of continuous functions $f: \R \to \R$, such that there exists $C > 0$ with $|f(x)| \leq C(1 + |x|)$ for all $x \in \R$. We write $\lambda$ for the Lebesgue measure on $[0,1]$.

\section{Biased probability measures and biased Strassen}\label{sec: biased_probabilities}

In this section we introduce and develop the framework for the $\beta$-biased order, a refinement of the classical convex order. In particular, we introduce $\beta$-biased probability measures as the basic building blocks of this structure, and prove our first main result---a Strassen-type characterization (Theorem \ref{thm:bbc-Strassen_intro}).

\subsection{Biased probability measures}\label{subsec:biased-probabilities}

We start by introducing the concept of $\beta$-biased probability measure and study its main properties.
In analogy to the convex order, where any probability measure on $\R$ with zero mean can be decomposed into a mixture of measures consisting of two atoms and centered at $0$, we define simple and atomic $\beta$-biased probabilities. These are required to have an atom of mass at least $\beta$ at the rightmost point of their support.

\begin{definition} \label{def.bb}
Let $\nu \in \Pc_1(\R)$ be centered at $0$ and $\beta \in (0,1)$. 
We say that
\begin{enumerate}
    \item $\nu$ is \emph{simple $\beta$-biased if $\nu|_{[0,\infty)}$ consists of a single atom with mass at least $\beta$};
    \item $\nu$ is \emph{atomic $\beta$-biased} if $S(\nu) < \infty$ and $\nu(\{S(\nu)\}) \geq \beta$;
    \item $\nu$ is \emph{$\beta$-biased} if $\nu$ is an average of atomic $\beta$-biased probabilities,
    i.e.\ $\nu = \int \nu_x \, \rho(dx)$, where $(\nu_x)_x$ are atomic $\beta$-biased probabilities and $\rho \in \Pc(\R)$. 
\end{enumerate} 
\end{definition} 

Note that if $\nu$ is $\beta$-biased, then $\tm{\nu^R} \geq \beta$. 

\begin{lemma}\label{lem:atomis_as_a_mixture_of_simple}
Every atomic $\beta$-biased probability is an average of simple $\beta$-biased probabilities. In particular, $\nu$ is $\beta$-biased if and only if $\nu$ is an average of simple $\beta$-biased probabilities.
\end{lemma}

\begin{remark}\label{rem:meausurability_of_decomposition}
    By Lemma \ref{lem:atomis_as_a_mixture_of_simple} every $\beta$-biased $\nu$ can be written in the form $\nu = \int \nu_x \, \rho(dx)$. It will be clear from the proof that we can take $\rho$ to be the Lebesgue measure $\lambda$ on $[0,1]$ and that the mapping $\nu \mapsto \lambda \otimes \nu_x \in \cpls(\lambda,\nu) $ can be taken to be Borel measurable.
\end{remark}

\begin{proof}
    Let $\nu \in \mathcal P_1(\R)$ be an atomic $\beta$-biased probability. Our goal is to directly find a representation of $\nu$ as a mixture of simple $\beta$-biased probabilities, i.e., we want to construct a kernel $(\nu_x)_{x \in \R}$ such that $\nu = \int \nu_x \, \rho(dx)$ for some $\rho \in \Pc(\R)$, and $\nu_x$ is simple $\beta$-biased for $\rho$-a.e. $x$.
    If $\nu = \delta_0$ there is nothing to show.
    Otherwise, $\bar \nu^L < 0$ and we define for $x \in [0,S(\nu)]$
    \begin{equation}\label{eq:decomposition_to_simple_atomic}
        B_x := \frac{- \bar{\nu}^L}{x - \bar{\nu}^L },\quad \rho(dx) := \frac{\nu^R(dx)}{B_x},\quad \nu_x := \frac{1-B_x}{\tm{\nu^L}} \nu^L + B_x \delta_x.
    \end{equation}
    Clearly, we have that $B_x \in [0,1]$ and $(\nu_x)_{x \in [0,S(\nu)]}$ is a  kernel on $\R$.
    To complete the proof, it remains to verify all of the following:
    \begin{enumerate}[label = (\roman*)]
        \item \label{it:L.atomis.proof.1} $\rho$ is a probability measure concentrated on $[0,S(\nu)]$;
        \item \label{it:L.atomis.proof.2} $\nu_x$ is simple $\beta$-biased for $x \in [0,S(\nu)]$;
        \item \label{it:L.atomis.proof.3} $\nu = \int \nu_x \, \rho(dx)$.
    \end{enumerate}
    To see \ref{it:L.atomis.proof.1}, it suffices to show that $\tm{\rho} = 1$. 
    Since $\bar{\nu} = 0$ we have that $\int y \, d\nu^R = - \tm{\nu^L} \bar{\nu}^L$ and thus
    \[
    \tm{\rho}  = \int \frac{x - \bar{\nu}^L}{-\bar{\nu}^L} \, \nu^R(dx) =  \frac{-\|\nu^L\|\bar \nu^L - \|\nu^R\| \bar \nu^L}{-\bar \nu^L} = 1.
    \]
    To get \ref{it:L.atomis.proof.2}, fix $x \in [0,S(\nu)]$.
    We have $\bar{\nu}_x = (1 - B_x) \bar{\nu}^L + B_x x = \bar{\nu}^L + B_x(x - \bar{\nu}^L) = 0$.
    Moreover, using that $\tm{\nu^R} \ge \nu(\{S(\nu)\}) \ge \beta$ since $\nu$ is atomic $\beta$-biased, $\bar \nu^L \le 0$ and $\|\nu^L\| \le 1 - \beta$, we get
    \[
        0 = \bar \nu = \bar \nu^L \tm{\nu^L} + \bar \nu^R \tm{\nu^R} \ge (1 - \beta) \bar \nu^L + \beta S(\nu),
    \]
    which yields
    \[
    \beta \leq \frac{-\bar \nu^L}{S(\nu) - \bar \nu^L} \leq \frac{-\bar \nu^L}{x - \bar \nu^L} = B_x.
    \]
    We have that $\nu_x(\{x\}) \ge \beta$ and therefore $\nu_x$ is simple $\beta$-biased.
    
    Finally, to see \ref{it:L.atomis.proof.3}, note that $\nu^R_x = B_x \delta_x$ and $\nu^L_x = \frac{1 - B_x}{\tm{\nu^L}} \nu^L$ for $x \in [0,S(\nu)]$, hence,
    \begin{align*}
        \int \nu_x^R \, \rho(dx) &= \int B_x \delta_x \, \rho(dx) = \int \delta_x \, \nu^R(dx) = \nu^R, \\
        \int \nu_x^L \, \rho(dx) &= \nu^L \int \frac{1-B_x}{\tm{\nu^L}} \, \frac{\nu^R(dx)}{B_x} = \nu^L \int \frac{x}{-\bar{\nu}^L\tm{\nu^L}} \, \nu^R(dx) = \nu^L.
    \end{align*}
    Thus, $\nu = \int \nu_x \, \rho(dx)$, proving \ref{it:L.atomis.proof.3} and thereby completing the proof.
\end{proof}

Our first goal is to characterize $\beta$-biased probabilities via testing against an appropriate class of functions. To this end, let
\[
\Psi := \{ f:\R \to \R : f(0)=0,\; f|_{\R_+}\ \text{convex},\; f\ \text{anti-symmetric, continuous} \}.
\]
For $\beta \in (0,1)$, let $a(\beta) = \frac{\beta}{1-\beta}$. For a function $f\colon \R \to \R$, we define the transformation $Z^\beta(f)\colon \R \to \R$ by 
$$
Z^\beta(f)(x) :=
\begin{cases}
    \frac{1}{a(\beta)} f(a(\beta)x), & \text{if } x > 0, \\
    f(x), & \text{if } x \leq 0.
\end{cases}
$$ 
Finally, we set $\Psi^\beta := \{Z^\beta(f) \mid f \in \Psi\}$.

\begin{figure}[H]
    \centering    \includegraphics[width=0.6\textwidth]{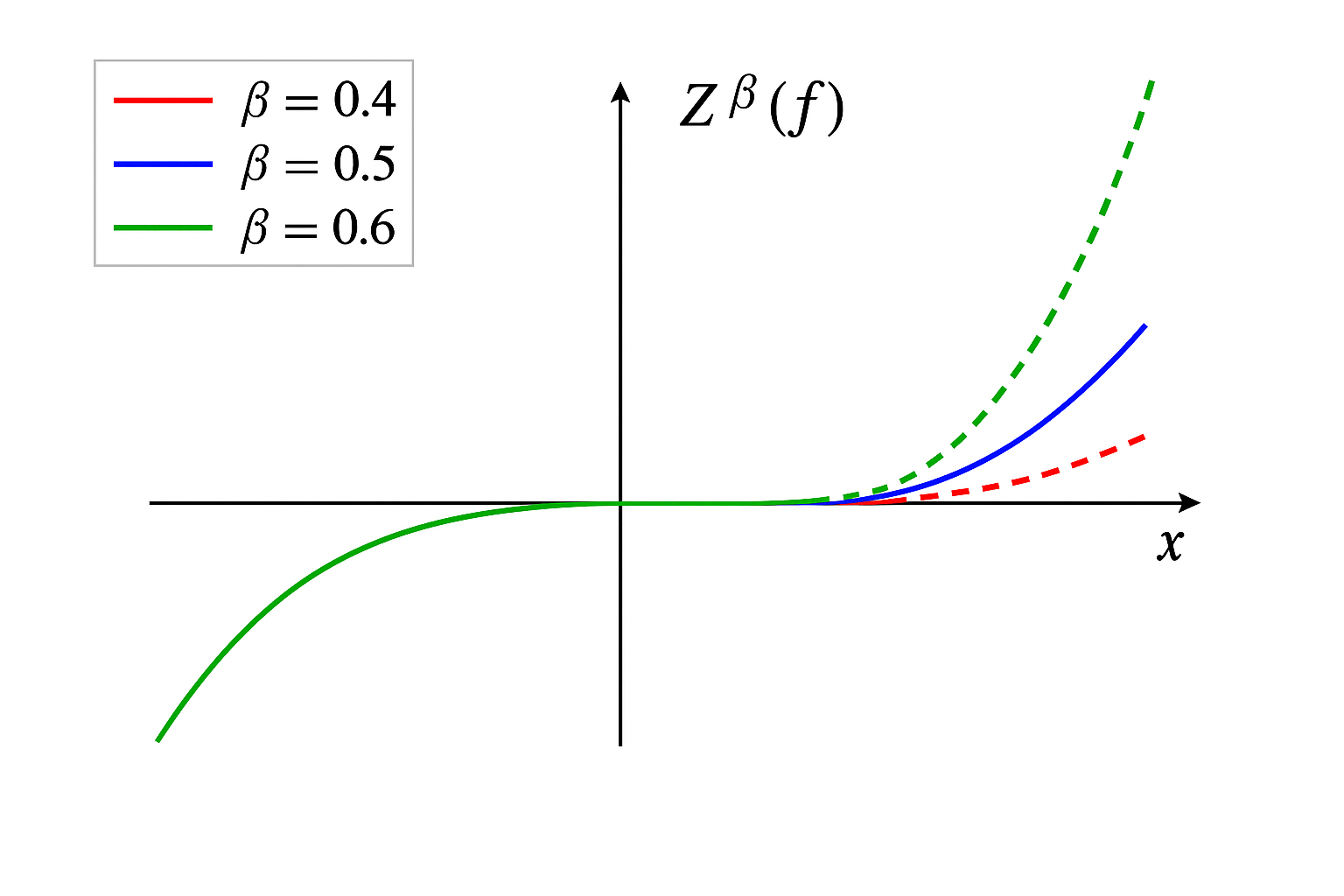}
    \caption{Graph of $f = Z^{1/2}(f) \in \Psi$ (blue) and its transformations $Z^\beta(f)$ for $\beta = 0.4 $ (red) and $\beta = 0.6$ (green).}
    \label{fig:Z_beta_transform}
\end{figure}

Note that, for any fixed $f \in \Psi$, the transform $Z^\beta(f)$ is increasing in $\beta$. Indeed, since $a(\cdot)$ is an increasing function and $f$ is convex on $\R_+$ with $f(0)=0$, it follows that, for any $\eta > \beta$ and $x \geq 0$, we have $f(a(\beta) x) \leq \frac{a(\beta)}{a(\eta)} f(a(\eta) x)$. Consequently, $Z^\beta(f) \leq Z^\eta(f)$.
Moreover, from the definition of \( \Psi^\beta \), it follows that its elements are antisymmetric up to a stretch and scaling of the positive axis. Specifically, for any $f \in \Psi^\beta$, we observe that 
$$
    f(x) = 
        \begin{cases}
            -\frac{1}{a(\beta)}\, f\big(R^\beta(x)\big) & x \geq 0 \\
            -a(\beta) \, f\big(R^\beta(x)\big) & x \leq 0
        \end{cases},
$$
where the distorted reflection function \( R^\beta(x) \) 
is defined as:
\begin{equation}\label{eq:Stratching Function}
     R^\beta(x) :=  \begin{cases}
        -a(\beta) x, & x \geq 0 \\
        - \frac{1}{a(\beta)} x, & x \leq 0.
    \end{cases} 
\end{equation}
For \( \beta = 1/2 \) (thus $a(\beta) = 1$), we have that $Z^{1/2}(f)=f$, so $\Psi^{1/2}=\Psi$ and $R^{1/2}(x)=-x$.

\smallskip

For convenience, for the rest of this section we fix parameters $\beta \in (0,1)$ and $ a= \frac{\beta}{1-\beta}$.

\begin{lemma}\label{lem:betabiased_test_functions}
    Let $\nu \in \Pc_1(\R)$. Then $\nu$ is $\beta$-biased if and only if $\nu(f) \leq 0$ for all $f \in \Psi^\beta$.
\end{lemma}

\begin{proof} 
    We first show the necessary condition. We start by considering a simple $\beta$-biased probability $\nu$ of the form $\nu = \nu^L + \beta \delta_{S(\nu)}$. For any $\phi \in \Psi$, due to Jensen's inequality, concavity of $\phi$ on $(-\infty,0]$, antisymmetry of $\phi$, and the fact that $(1-\beta) \bar \nu^L + \beta S(\nu) = 0$, we have that 
    $$
    \nu\big(Z^\beta (\phi)\big) = \nu^L(\phi) + (1-\beta)\phi(aS(\nu)) \leq - (1 - \beta) \phi(- \bar \nu^L) + (1-\beta) \phi(a S(\nu) ) = 0.
    $$
    Since $\phi \in \Psi$ was  arbitrarily, we conclude that $    \nu(f) \leq 0
    $
    for all $f \in \Psi^\beta$.
    
Next, if $\nu$ is a simple $\beta$-biased probability of the form $\nu = \nu^L + \gamma \delta_{S(\nu)}$ for some $\gamma \geq \beta$, then in the same way as above we can see that, for any $\phi \in \Psi$, we have $\nu\big(Z^\gamma (\phi)\big) \leq 0$. Since $Z^\beta (\phi) \leq Z^\gamma (\phi)$, we obtain that $\nu\big( Z^\beta (\phi) \big) \leq 0.$ 

Finally, let $\nu$ be any $\beta$-biased probability. Then, by Lemma \ref{lem:atomis_as_a_mixture_of_simple}, $\nu$ can be represented as a mixture of simple $\beta$-biased probabilities $\nu = \int \nu_x \, \rho(dx)$. For $\rho$-a.e. $x$, we have $\nu_x(f) \leq 0$ for all $f \in \Psi^\beta$. Therefore, $\nu(f) \leq 0$ for all $f \in \Psi^\beta$. 

    We are now going to show the reverse implication. For this, assume that $\nu(f) \leq 0$ for all $f \in \Psi^\beta$, or equivalently, 
    \begin{equation}\label{eq:1}
        \nu^R(f) \leq \nu^L(-f). 
    \end{equation}
    Our goal is to explicitly represent $\nu$ as a mixture of simple $\beta$-biased probabilities. We proceed in two steps. First, we show that \eqref{eq:1} ensures the existence of a martingale coupling between $\frac{\nu^R}{\tm{\nu^R}}$ and $\frac{a R^\beta_{\#} \nu^L + \alpha \delta_0}{\tm{\nu^R}}$ with some $\alpha \geq 0$. Then, we construct the required representation using the kernel corresponding this martingale coupling.

    Fix $f \in \Psi^\beta$, and consider $x \leq 0$. Then, by  definition of $\Psi^\beta$, we have $f(x) = -a f(R^\beta(x))$. Substituting this into \eqref{eq:1}, we get 
    $$
    \nu^R(f) \leq \int a f(R^\beta(x)) \, \nu^L(dx) = a \int f(x) \, R^\beta_{\#}\nu^L(dx) = a R^\beta_{\#}\nu^L(f).
    $$    
    Since the set of functions in $\Psi^\beta$ restricted to $\mathbb{R}_+$ is equal to the set of convex functions $\phi$ on $\mathbb{R}_+$ such that $\phi(0) = 0$, inequality \eqref{eq:1} can be rewritten as
    \begin{equation}\label{eq:convex order ineq}
        \nu^R(\phi) \leq a R^\beta_{\#}\nu^L(\phi) 
    \end{equation}
    for all convex functions $\phi$ on $\mathbb{R}_+$ such that $\phi(0)=0$. 
    Testing \eqref{eq:convex order ineq} on the functions 
    $$
            \phi_\varepsilon(x) := 
    \begin{cases} 
    -1, & x\geq \varepsilon \\
    -\frac{x}{\varepsilon}, & x \in [0, \varepsilon]
    \end{cases},
    $$
    we find that $\tm{\nu^R} \geq a \tm{\nu^L}$. Indeed, using that $\frac{x}{\varepsilon} \leq 1$ for  $x \in [0,\varepsilon]$ and applying \eqref{eq:convex order ineq} with $\phi = \phi_\varepsilon$, we obtain
    \begin{align*}
    - \tm{\nu^R} &\leq \int_{[0,\varepsilon)} -\frac{x}{\varepsilon} \nu^R(dx) - \nu^R([\varepsilon, +\infty)) \\  
    &\leq a \int_{[0,\varepsilon)} -\frac{x}{\varepsilon} \, R^\beta_{\#}\nu^L(dx) - a R^\beta_{\#}\nu^L([\varepsilon, +\infty)) \leq - a R^\beta_{\#}\nu^L([\varepsilon, +\infty)).
    \end{align*}
    Letting $\varepsilon \to 0$, we obtain $a\tm{\nu^L} \leq \tm{\nu^R}$. Thus, \eqref{eq:convex order ineq} implies that 
    \begin{equation}\label{eq:nuR leqC nuL}
        \nu^R \prec a R^\beta_{\#} \nu^L + \alpha \delta_0,
    \end{equation}
    where $\alpha := \tm{\nu^R} - a\tm{\nu^L} \geq 0$. 
    Hence, by \cite[Theorem 8]{St65}, there exists a coupling
    $$
    \pi \in \cpls \left(\frac{\nu^R}{\tm{\nu^R}}, \frac{aR^\beta_{\#} \nu^L + \alpha \delta_0}{\tm{\nu^R}} \right)
    $$
    such that $\bar{\pi}_x = x$ $\nu^R$-a.e., where $(\pi_x)_x$ denotes the disintegration of $\pi$ with respect to the first marginal.
    This leads to
    $$
    aR^\beta_{\#} \nu^L + \alpha \delta_0 = \int \pi_x \, \nu^R(dx),
    $$
    implying, using that $R^\beta$ is an involution,
    $$
    \nu^L = \frac{1}{a}\int  R^\beta_{\#} \pi_x \, \nu^R(dx) - \frac{\alpha}{a} \delta_0,
    $$
    and comparing mass at $0$ on both sides yields
    $\alpha = \int \pi_x(\{0\}) \, \nu^R(dx)$.
    Since $\nu^R = \int \delta_x\,  \nu^R(dx)$, we obtain
\begin{equation}\label{eq:Decomposition_nu_v1}
    \nu =  \nu^L + \nu^R = \int \left(\frac{1}{a} (R^\beta_{\#}  \pi_x  - \pi_x(\{0\})\delta_0) + \delta_x \right) \nu^R(dx).
\end{equation}  
    Denoting $\beta_x := \frac{1}{a} + 1 - \frac{1}{a} \pi_x(\{0\})$ and setting
    \begin{align}\label{eq:Decomposition_nu_v2+}
    \nu_x := \frac{\frac{1}{a} (R^\beta_{\#}  \pi_x  - \pi_x(\{0\})\delta_0) + \delta_x}{\beta_x}, \quad \mu(dx) := \beta_x \,\nu^R(dx),
\end{align}
we can rewrite \eqref{eq:Decomposition_nu_v1} as
\begin{equation}\label{eq:Decomposition_nu_v2}
    \nu = \int \nu_x \, \mu(dx).
\end{equation}   

    The final step consists in showing that $(\nu_x)_x$ in \eqref{eq:Decomposition_nu_v2+} are simple $\beta$-biased probabilities and that $\mu$ is a probability measure.
    Since $\alpha = \tm{\nu^R} - a\tm{\nu^L} = \tm{\nu^R} - a(1-\tm{\nu^R})$, 
    $$
    \int \beta_x \, \nu^R(dx) = (1/a + 1)\tm{\nu^R} - \frac{\alpha}{a} = 1,
    $$ 
    and hence $\mu$ is a probability measure.     
    From the definition of $\beta_x$, it follows that $\tm{\nu_x} = 1$, and since the barycenter of $R^\beta_{\#}\pi_x$ is $-ax$, $\bar{\nu}_x = 0$, so $\nu_x$ are probability measures centered at $0$. Notice that, for $\nu^R$-a.e.\ $x$, the support of $\pi_x$ is concentrated on $\mathbb{R}_+$, hence the support of $R^\beta_{\#}  \pi_x$ is concentrated on $\mathbb{R}_-$. Hence, $\nu_x|_{[0,\infty)} = \frac{1}{\beta_x} \delta_x$ is a single atom with mass 
    \begin{equation}\label{eq:atoms_in_decomposition}
          \frac{1}{\beta_x} = \frac{1}{1/a + 1 - \pi_x(\{0\})/a} \geq \frac{1}{1/a + 1} = \beta.
    \end{equation}
    Thus, $(\nu_x)_x$ are simple $\beta$-biased probabilities and $\nu$ can be represented as in \eqref{eq:Decomposition_nu_v2}, implying that $\nu$ is $\beta$-biased.
\end{proof}

The following corollary establishes a connection between convex order and the concept of $\beta$-biased probabilities. It shows that, for $\nu \in \Pc_1(\R)$, the property of being $\beta$-biased is equivalent to the structural condition that $\nu^R$ is in convex order with a certain distorted reflection of $\nu^L$. 

\begin{corollary}\label{cor:convexbiased}
Let $\nu \in \Pc_1(\R)$. The following are equivalent:
\begin{enumerate}
    \item $\nu$ is $\beta$-biased;
    \item $\nu(f) \leq 0$ for all $f \in \Psi^\beta \cap C_L(\R)$;
    \item $\nu^R \prec a R^\beta_{\#} \nu^L + \alpha \delta_0$, where $\alpha = \tm{\nu^R} - a\tm{\nu^L} \geq 0$.
\end{enumerate}
\end{corollary}

\begin{proof}
(1) $\Rightarrow$ (2): If $\nu$ is $\beta$-biased, $\nu(f) \leq 0$ for all $f \in \Psi^\beta$ by Lemma~\ref{lem:betabiased_test_functions}, hence also for all $f \in \Psi^\beta \cap C_L(\R)$.

(2) $\Rightarrow$ (3): Assume that $\nu(f) \leq 0$ for all $f \in \Psi^\beta \cap C_L(\R)$. Then, as in the proof of Lemma~\ref{lem:betabiased_test_functions}, we obtain
\begin{equation}\label{eq:betabiased_test_functions_linear}
    \nu^R(\phi) \leq a R^\beta_{\#} \nu^L(\phi)  
\end{equation}
for all convex functions $\phi$ on $\R_+$ such that $\phi(0) = 0$ and $\phi \in C_L(\R)$, as well as $\tm{\nu^R} \geq a |\tm{\nu^L}|$. Testing \eqref{eq:betabiased_test_functions_linear} on functions $\phi_k(x) := (k - x)_+$ for $k\in\R$, we deduce that
$$
\nu^R \prec a R^\beta_{\#} \nu^L + \alpha \delta_0,
$$
where $\alpha = \tm{\nu^R} - a \tm{\nu^L} \geq 0$.

(3) $\Rightarrow$ (1): Assume $\nu^R \prec a R^\beta_{\#} \nu^L + \alpha \delta_0$ with $\alpha \geq 0$. For any $f \in \Psi_\beta$, by definition, there exists a convex function $\phi$ on $\mathbb{R}_+$ such that 
$$
f(x) = 
\begin{cases}
    \frac{1}{a} \phi(ax), & x \geq 0, \\
    -\phi(-x), & x \leq 0.
\end{cases}
$$
Now, define $\psi(x) := \frac{1}{a} \phi(ax)$. From the assumption it follows that $\nu^R(\psi) \leq a R^\beta_{\#} \nu^L(\psi)$, which implies that $\nu(f) \leq 0$. Since $f \in \Psi_\beta$ was arbitrary, Lemma~\ref{lem:betabiased_test_functions} ensures that $\nu$ is $\beta$-biased.
\end{proof}

So far we have focused on the case where $\nu$ is centered in $0$. In order to consider the case of general barycenters, we introduce the following definition: 
\begin{definition}\label{Def: beta biased around x}
Consider the translations $T_x(y):=y+x$, $x \in \R$. A measure $\nu$ is called \emph{$\beta$-biased around $x$} if it is obtained as a shift of a $\beta$-biased measure, that is, if $\nu = T_{x \#} \nu_0$ for some $\beta$-biased $\nu_0$. We also set $\Psi_x := \{f \circ T_{-x} : f \in \Psi\}$ and $\Psi_x^\beta := \{f \circ T_{-x} : f \in \Psi^\beta\}$.
\end{definition}

By analogy with \eqref{eq:Stratching Function}, we define the following mapping: 
\begin{equation}\label{eq:R function 2}
    R^{\beta,x}(y) := 
    \begin{cases}
        x - \frac{y-x}{a(\beta)}, & \text{if } y \leq x,\\
        x - a(\beta)(y-x), & \text{if } y \geq x,
    \end{cases}
\end{equation}
so that $R^{\beta,0}=R^\beta$.
Lemma~\ref{lem:betabiased_test_functions} and Corollary \ref{cor:convexbiased} then extend to the case of general barycenters in the obvious way: 

\begin{corollary}\label{cor:Characterization_biased_around_x}
Let $\nu \in \Pc_1(\R)$. The following are equivalent:
\begin{enumerate}
    \item $\nu$ is $\beta$-biased around $x$;
    \item $\nu(f)\leq 0$ for all $f\in \Psi_x^\beta$;
    \item $\nu^R \prec a R^{\beta,x}_{\#} \nu^L + \alpha \delta_x,$ with $\alpha = \tm{\nu^R} - a\tm{\nu^L} \geq 0$.
\end{enumerate}
\end{corollary}

\begin{remark}\label{rem:symm}
The last corollary implies that any $\nu \in \mathcal{P}_1(\mathbb{R})$ that is symmetric with respect to its barycenter $x = \bar{\nu}$, in the sense that
$
\nu^R = R^{1/2,x}_{\#} \nu^L,
$
is $1/2$-biased, but not $\beta$-biased for any $\beta > 1/2$, since $\tm{\nu^R} - a(\beta)\tm{\nu^L} < 0$ for $\beta > 1/2$.
\end{remark}

\subsection{Biased order and biased martingale couplings}\label{subsec:biased-strassen}
In this section we study the $\beta$-biased order and $\beta$-biased couplings, and present the announced Strassen-type result.
\begin{definition}\label{def:bbc-coupling}
Let $\mu, \nu \in \Pc_1(\R)$ and $\beta\in[0,1)$. The set of $\beta$-biased martingale couplings 
of $\mu$ and $\nu$, denoted by $\cplbbc(\mu, \nu)$, consists of all couplings $\pi$ of $\mu$ and $\nu$ such that the $\mu$-disintegration $(\pi_x)_x$ of $\pi$ satisfies $\mu$-a.s. that $\pi_x$ is $\beta$-biased around $x$. 
\end{definition}

Note that $\cpls_0(\mu, \nu)$ coincides with the set of martingale couplings between $\mu$ and $\nu$. 
We first prove that couplings in $\cplbbc(\mu, \nu)$ correspond to laws of  $\beta$-biased martingales defined in the introduction.

\begin{lemma}\label{lem:beta_biased_martingale}
    Let $\mu,\nu \in \Pc_1(\R)$,  $\pi\in \cpls(\mu,\nu)$, and $\beta \in (0,1)$.  Then $\pi \in \cplbbc(\mu,\nu)$ if and only if there exists a $\beta$-biased martingale $(X_t)_{t=0,1}$ defined on some stochastic basis $(\Omega,\mathcal{F},\P,(\F_t)_{t=0,1})$ with $\law(X_0,X_1) = \pi$.
\end{lemma}

\begin{proof}
    Let $\pi \in \cplbbc(\mu,\nu)$ and let $(\pi_x)_x$ be its $\mu$-disintegration. Since for $\mu$-a.e. $x$, $\pi_x$ is $\beta$-biased around $x$, by Lemma \ref{lem:atomis_as_a_mixture_of_simple} we can decompose  
    $\pi_x$ into a mixture of simple $\beta$-biased probabilities $(\pi_{x,u})_u$, so that $\pi_x = \int \pi_{x,u} \, \lambda(du).
    $
    Note that, by  Remark~\ref{rem:meausurability_of_decomposition}, $(\pi_{x,u})_{x,u}$ can be taken to be measurable in $(x,u)$.

    Define the probability space $(\Omega,\F,\P)$ by taking $\Omega = \R^3$, $\F = \B(\R^3)$, and setting 
    $$ \P(dx, du, dy) = \mu(dx) \otimes \lambda(du) \otimes \pi_{x,u}(dy).$$
    On this space, define the random variables  
    $$
    X_0(x,u,y) = x, \quad U(x,u,y) = u, \quad X_1(x,u,y) = y.
    $$ 
    Finally, set the filtration $\F_0 = \sigma(X_0,U)$ and $\F_1 = \sigma(X_0,U,X_1)$. By construction, $(X_0,X_1)$ is a martingale. Moreover, since $\law(X_1|X_0=x,U=u) = \pi_{x,u}$, it follows that $X$ is a $\beta$-biased martingale.

    Conversely, let $(X_t)_{t=0,1}$ be a $\beta$-biased martingale. Then, for all $A \in \B(\R)$ we have
    \[
    \P(X_1 \in A | X_0) = \E[\P(X_1 \in A | \F_0) | X_0],
    \]
    which implies that $\law(X_1|X_0)$ is a.s.\ $\beta$-biased around $X_0$. Finally, from $\law(X_0,X_1) = \pi \in \cpls(\mu,\nu)$, we conclude that $\pi \in \cplbbc(\mu,\nu)$.
\end{proof}

We recall the notion of $\beta$-biased order given in the introduction, as it is central to our study. For this we use the $\beta$-envelope $g_\beta$ defined in \eqref{eq:beta_transform}. 
\begin{definition}
Let $\beta\in[0,1)$. We say that $\mu, \nu\in \Pc_1(\R)$ are in $\beta$-biased order, denoted $\mu\bbc\nu$, if $\mu(g_\beta) \leq \nu(g)$ for all $g \in C_{L}(\R)$. In particular, for $\beta = 0$, the $\beta$-biased order coincides with the convex order.
\end{definition}

\begin{remark}\label{rem:g_beta_transform}
For any $g \in C_L(\mathbb{R})$ and $x \in \mathbb{R}$, the $\beta$-envelope satisfies
$$
g_\beta(x) = \inf \{ \nu(g) : \delta_x \prec_\beta \nu \}.
$$
Indeed, if $\delta_x \prec_\beta \nu$, then $\nu$ admits a decomposition $\nu = \int \nu_u \, \rho(du)$ such that $\bar{\nu}_u = x$ and $\nu_u(\{S(\nu_u)\}) \geq \beta$ for $\rho$-almost every $u$, so 
$
\nu(g) = \int \nu_u(g)\, \rho(du) \geq \int g_\beta(x)\, \rho(du) = g_\beta(x),
$
hence $g_\beta(x) \leq \inf \{ \nu(g) : \delta_x \prec_\beta \nu \}$. The reverse inequality holds by definition of $g_\beta(x)$.
\end{remark}

\begin{proof}[Proof of Theorem~\ref{thm:bbc-Strassen_intro}]
If $\beta = 0$, the claim reduces to the Strassen theorem for convex order. We therefore assume $\beta \in (0,1)$ in what follows.

Define $c: \mathbb{R} \times \mathcal{P}_1(\mathbb{R}) \to \mathbb{R} \cup \{+\infty\}$ by
$$
c(x,p) := \sup_{f \in \Psi^\beta \cap C_L(\mathbb{R})} \int f(y - x)\, p(dy).
$$
Then $c$ is convex and lower semi-continuous in  $p$ (as a supremum of linear and continuous functions $p\mapsto\int f(y-x)\, p(dy)$), and bounded from below (since $0 \in \Psi^\beta$). 
By Corollary~\ref{cor:convexbiased}, we have
$$
c(x,p) = 
\begin{cases}
0, & \text{if } p \text{ is } \beta\text{-biased around } x, \\
+\infty, & \text{otherwise}.
\end{cases}
$$

Consider the weak optimal transport problem
$$
T_c(\mu, \nu) := \inf_{\pi \in \cpls(\mu, \nu)} \int c(x, \pi_x)\, \mu(dx),
$$
where $(\pi_x)_x$ is the $\mu$-disintegration of $\pi$.  
It follows that $T_c(\mu, \nu) = 0$ if and only if $\cplbbc(\mu, \nu) \neq \emptyset$.

For any $\psi \in C_L(\mathbb{R})$, we have
$$
\inf_{p \in \mathcal{P}_1(\mathbb{R})} \left( p(\psi) + c(x,p) \right) = \psi_\beta(x),
$$
and by the weak transport duality \cite[Theorem~3.1]{BaBePa18}, it follows that
$$
T_c(\mu, \nu) = \sup_{\psi \in C_L(\mathbb{R})} \left( \mu(\psi_\beta) - \nu(\psi) \right).
$$
Thus, $T_c(\mu, \nu) = 0$ if and only if $\mu(\psi_\beta) \leq \nu(\psi)$ for every $\psi \in C_L(\mathbb{R})$.
    Hence, we obtain the equivalence:
\[
    \mu \prec_{\beta} \nu \quad \text{if and only if} \quad \cplbbc(\mu, \nu) \neq \emptyset.
\]
\end{proof}
Note that Theorem \ref{thm:bbc-Strassen_intro} in particular implies that
\[
\text{$\nu$ is $\beta$-biased around $x$ \quad $\Longleftrightarrow$ \quad $\delta_x \bbc \nu$},
\]
in which case $x=\bar \nu$.

 The following proposition provides another useful representation of the $\beta$-envelope $g_\beta$.
 
 \begin{proposition}
 The $\beta$-envelope of $g \in C_L(\R)$ can be written as
   \begin{equation}
       g_\beta(x) = \sup\left\{y : \exists f\in \Psi_x^\beta \text{ such that } y\leq f+g \right\}.
   \end{equation}
 \end{proposition}
 \begin{proof}
     First note that, for any $f \in \Psi_x^\beta$ and $\nu$ such that $\delta_x \bbc \nu$, we have $\nu(f) \leq 0$ and $\sup_{f \in \Psi_x^\beta} \nu(f) = 0$. Then, from \eqref{eq:beta_transform}, we get
    $$ g_\beta(x) = \inf_{\delta_x \bbc \nu} \sup_{f \in \Psi_x^\beta} \nu(f + g).$$
     Moreover, for any probability $\rho \notin  \{\nu: \delta_x \bbc \nu \}$, we have that $\sup_{f \in \Psi_x^\beta} \rho(f) = +\infty$, hence 
     \begin{equation}\label{eq:g_beta_via_int}
         g_\beta(x) = \inf_{\nu \in \mathcal{P}(\R)} \sup_{f \in \Psi_x^\beta} \nu(f + g).
    \end{equation}

     On the other hand, 
     \begin{align*}
    \sup\Bigl\{y : \exists f \in \Psi_x^\beta \text{ such that } y \le f+g \Bigr\}
    &= \sup_{f \in \Psi_x^\beta} \inf_{y \in \R} \bigl(f(y) + g(y)\bigr) \\
    &= \sup_{f \in \Psi_x^\beta} \inf_{\nu \in \mathcal{P}(\R)} \nu(f + g).
    \end{align*}
     Thus, we are left to prove the following min-max equality: 
     $$ \inf_{\nu \in \mathcal{P}(\R)} \sup_{f \in \Psi_x^\beta} \nu(f + g) = \sup_{f \in \Psi_x^\beta} \inf_{\nu \in \mathcal{P}(\R)} \nu(f + g).$$
     For $n\in\mathbb{N}$, let $K_n := [x-n, x+n]$ and denote by $\mathcal P (K_n)$ the set of probability measures supported on it. Then, by Prokhorov's theorem, $\mathcal P (K_n)$ is compact, hence the min-max theorem yields
     \begin{equation}\label{eq:minimax_prop_1sec}
      \inf_{\nu \in \mathcal{P}(K_n)} \sup_{f \in \Psi_x^\beta} \nu(f + g) = \sup_{f \in \Psi_x^\beta} \inf_{\nu \in \mathcal{P}(K_n)} \nu(f + g).
    \end{equation}
    From Lemma \ref{lem:atomis_as_a_mixture_of_simple},  every $\beta$-biased $\nu \in \Pc_1(\R)$ can be approximated in $\mathcal W_1$ by a sequence $(\nu_n)_n$ of compactly supported $\beta$-biased probabilities. Using Remark \ref{rem:g_beta_transform}, it is immediate that the sequence $\inf_{\nu \in \mathcal{P}(K_n)} \sup_{f \in \Psi_x^\beta} \nu(f + g)$ is monotonically decreasing to $g_\beta(x)$, hence we can pass to the limit in \eqref{eq:minimax_prop_1sec} and obtain
      \[ g_\beta(x) = \inf_{\nu \in \mathcal{P}(\R)} \sup_{f \in \Psi_x^\beta} \nu(f + g) = \sup_{f \in \Psi_x^\beta} \inf_{\nu \in \mathcal{P}(\R)} \nu(f + g).\]
\end{proof}

\section{Strongly biased probability measures}\label{sec: strongly_biased_probabilities}
In this section, motivated by the arbitrage-free pricing of American options (see argument in the Introduction), we define a stronger version of the $\beta$-biased order and examine its properties.
For the remainder of this section, we fix a parameter  $\beta \in (0,1)$   and write $a := a(\beta) = \frac{\beta}{1-\beta}$
for convenience. This notation will be used unless otherwise specified.

\subsection{Strongly biased probability measures}

\begin{definition}\label{def:str_biased}
A measure $\nu \in \Pc_1(\R)$ centered at $0$ is called \emph{strongly $\beta$-biased} if there exist a probability measure $\rho \in \Pc(\R)$ and a kernel $(\nu_x)_x \subset \Pc_1(\R)$, where $\nu_x$ is simple $\gamma(x)$-biased 
with $\gamma(x) > \beta$ 
such that $\nu = \int \nu_x \, \rho(dx)$.

A measure $\nu \in \Pc_1(\R)$ is called \emph{strongly $\beta$-biased around $x$} if it is obtained as a shift of a strongly $\beta$-biased measure, i.e., if $\nu = T_{x \#} \nu_0$ for some strongly $\beta$-biased $\nu_0$.
\end{definition}

\begin{lemma}\label{lem:atomic_not_simple_implies_stronglybiased}
Let $\nu$ be an atomic $\beta$-biased probability. If $\tm{\nu^R} > \beta$, then $\nu$ is strongly $\beta$-biased.
\end{lemma}
Note that the lemma in particular implies that any atomic $\beta$-biased probability which is not simple is strongly $\beta$-biased.
\begin{proof}
    From $$0 = \bar{\nu} = \bar{\nu}^L \tm{\nu^L} + \bar{\nu}^R \tm{\nu^R} > \beta S(\nu) + (1 - \beta)\bar{\nu}^L,$$ it follows that $\beta < \frac{-\bar{\nu}^L}{x - \bar{\nu}^L} =: B_x$,
    for all $x \in [0, S(\nu)]$. By \eqref{eq:decomposition_to_simple_atomic}, we then have the representation $\nu = \int \nu_x \, \rho(dx)$, where $(\nu_x)_{x\geq0}$ are $B_x$-biased probabilities with $B_x > \beta$.
\end{proof}

\begin{lemma}\label{lem:support^Restriction0}
Let $\nu \in \Pc_1(\R)$ be $\beta$-biased around $x$ and assume that $|s(\nu)| < \infty$. Then  
$$S(\nu) \leq \frac{x}{\beta} - \frac{s(\nu)}{a}.$$  
Furthermore, if $\nu \neq \delta_x$ and $\nu$ is strongly $\beta$-biased around $x$, then  
$\nu\big(\big\{ \frac{x}{\beta} - \frac{s(\nu)}{a} \big\}\big) = 0.$
\end{lemma}

\begin{proof}
    Let $\nu = \int \nu_u \, \rho(du)$ be a decomposition of $\nu$ as a mixture of simple  $\beta$-biased probabilities $(\nu_u)_u$, with $\rho$ being a probability measure. Then,  
    $$x = \bar{\nu}_u = \bar{\nu}_u^L \tm{\nu_u^L} + \tm{\nu_u^R} \bar{\nu}_u^R \geq (1-\beta) s(\nu_u) + \beta S(\nu_u),$$  
    implying  
    \begin{equation}\label{eq:bounds_of_biased_support}
    S(\nu_u) \leq \frac{x}{\beta} - \frac{s(\nu_u)}{a} \leq \frac{x}{\beta} - \frac{s(\nu)}{a}.
    \end{equation}  
    Since $S(\nu) \leq \sup_u S(\nu_u)$, from \eqref{eq:bounds_of_biased_support} we get $S(\nu) \leq \frac{x}{\beta} - \frac{s(\nu)}{a}$.  
    Moreover, we have equality in \eqref{eq:bounds_of_biased_support} if and only if $\nu_u = \beta \delta_{S(\nu_u)} + (1-\beta) \delta_{s(\nu)}$. Therefore, if $(\nu_u)_u$ are simple strongly $\beta$-biased around $x$ and $\nu \neq \delta_x$, we have a strict inequality in \eqref{eq:bounds_of_biased_support}, implying 
    $\nu_u(\{ \frac{x}{\beta} - \frac{s(\nu)}{a} \} ) = 0$ $\rho$-a.e.,  
    and so  
    $\nu(\{ \frac{x}{\beta} - \frac{s(\nu)}{a} \} ) = 0.$
\end{proof}

In particular, Lemma~\ref{lem:support^Restriction0} can be used to construct distributions in $\Pc_1(\R)$ that are not $\beta$-biased around their center for any $\beta > 0$. The following corollary provides a sufficient condition for this to occur.
\begin{corollary}
\label{cor:not_beta_biased}
    Let $\nu \in \Pc_1(\R)$ satisfy $|s(\nu)| < \infty$ and $S(\nu) = \infty$. Then $\nu$ is not $\beta$-biased around $x = \bar{\nu}$ for any $\beta > 0$.
\end{corollary}

Following \cite{BeJu16}, we define a stricter version of convex order.
\begin{definition}\label{def: irreducibility}
    Let $\mu, \nu$ be finite measures on $\mathbb{R}$ with finite first moment. 
    We say that the pair $(\mu, \nu)$ is irreducible if $\mu \prec \nu$ and one of the following holds true:
    \begin{enumerate}[label = (\roman*)]
        \item there exists an open interval $I$ (bounded or not) such that $\mu(I)$ and $\nu(\overline{I})$ have the total mass, and $\int(k-y)_+d \mu < \int(k-y)_+d \nu $ for all $k \in I$;
        \item $\mu = \nu$ is a Dirac measure.
    \end{enumerate}
    We write $\mu \suc \nu$ if $(\mu,\nu)$ is irreducible.

\end{definition}

The following lemma generalizes the results presented in Corollary \ref{cor:convexbiased}  to the case of strongly $\beta$-biased measures, and shows the difference between these two notions.

\begin{lemma}\label{L: stronglybetabiased_characterization}
Let $\nu \in \Pc_1(\R)$. Then:
\begin{enumerate}
    \item $\nu$ is strongly $\beta$-biased if and only if there exists a coupling
    $$
    \pi \in \cpls\left( \frac{\nu^R}{\tm{\nu^R}}, \frac{a R^\beta_{\#}\nu^L + \alpha \delta_0}{\tm{\nu^R}} \right),
    $$
    with $\alpha = \tm{\nu^R} - a\tm{\nu^L} > 0$, such that for $\nu^R$-almost all $x$, $\bar{\pi}_x = x$ and $\pi_x(\{0\}) > 0$, where $(\pi_x)_x$ is the disintegration of $\pi$ with respect to the first marginal.
    
    \item If $\nu(\{0\}) = 0$, then $\nu$ is strongly $\beta$-biased if and only if  
    $$
    \nu^R \suc aR^\beta_{\#} \nu^L + \alpha \delta_0,
    $$
    with $\alpha = \tm{\nu^R} - a\tm{\nu^L} > 0$.
\end{enumerate}
\end{lemma}

\begin{proof}
First, we show that for a measure $\nu$ to be strongly $\beta$-biased, it is necessary that $\alpha = \tm{\nu^R} - a\tm{\nu^L} > 0$. If $\nu$ is strongly $\beta$-biased, then from Corollary \ref{cor:convexbiased} it follows that
$$\nu^R \prec a R^{a}_{\#} \nu^L + \alpha \delta_0,$$
with $\alpha \geq 0$. Thus, $\alpha = \tm{\nu^R} - a\tm{\nu^L} = \frac{\tm{\nu^R} - \beta}{1-\beta}$, and since $\nu$ is a mixture of probabilities each with an atom larger than $\beta$ on $\mathbb{R}_+$, this implies $\tm{\nu^R} - \beta > 0$, thus ensuring $\alpha > 0$.

      Ad 1. Suppose that $\nu$ is strongly $\beta$-biased. If $\nu$ is simple, meaning that $\nu = \nu^L + \gamma \delta_M$ for some $\gamma > \beta$ and $M\geq0$, then the set $\cpls\left( \frac{\nu^R}{\tm{\nu^R}}, \frac{a R^\beta_{\#} \nu^L + \alpha \delta_0}{\tm{\nu^R}} \right)$ contains a unique coupling, given by  
        $$
        \pi = \delta_M \otimes  \frac{a R^\beta_{\#}\nu^L + \alpha \delta_0}{\gamma}.
        $$  
        The corresponding disintegration kernel with respect to the first marginal consists of a single probability measure $
        \pi_M = \frac{a R^\beta_{\#}\nu^L + \alpha \delta_0}{\gamma}.
        $
        This yields $\bar{\pi}_M = M$ and $\pi_M(\{0\}) = \frac{\alpha}{\gamma}=\frac{\gamma-\beta}{\gamma(1-\beta)} > 0$.

        Let now $\nu$ be a general strongly $\beta$-biased probability, and consider its decomposition $\nu = \int \nu_u \, \rho(du)$, with $\nu_u$ simple $\gamma(u)$-biased probability, for some $\gamma(u) > \beta$. As established,  for $\rho$-a.e. $u$, 
        the set $\cpls\left( \frac{\nu_u^R}{\tm{\nu^R_u}}, \frac{a R^\beta_{\#} \nu_u^L + \alpha_u \delta_0}{\tm{\nu^R_u}} \right)$ consists of a unique coupling $\pi^u$ satisfying $\pi^u_x(\{0\}) > 0$ and $\bar{\pi}_x^u = x$ for $\nu_u^R$-almost every $x$.
Setting $\pi := \int \pi^u \, \rho(du)$, we have that $\pi \in \cpls\left( \frac{\nu^R}{\tm{\nu^R}}, \frac{a R^\beta_{\#} \nu^L + \alpha \delta_0}{\tm{\nu^R}} \right)$,  
        and its disintegration  with respect to the first marginal $(\pi_x)_x$ satisfies $\bar{\pi}_x = x$ and $\pi_x(\{0\}) > 0$ for $\nu^R$-almost all $x$.

        Now suppose that there exists a coupling $\pi \in \cpls\left( \frac{\nu^R}{\tm{\nu^R}}, \frac{a R^\beta_{\#}\nu^L + \alpha \delta_0}{\tm{\nu^R}} \right)$ such that $\bar{\pi}_x = x$ and $\pi_x(\{0\}) > 0$ for $\nu^R$-almost all $x$. Repeating the proof of Lemma \ref{lem:betabiased_test_functions} and using \eqref{eq:Decomposition_nu_v2} and \eqref{eq:atoms_in_decomposition}, we obtain that $\nu$ admits the decomposition $\nu = \int \nu_x \, \mu(dx)$, where $\mu\in \Pc(\R)$ and each $\nu_x$ satisfies $\bar{\nu}_x = 0$ and  
        $
        \nu_x^R = \frac{1}{\beta_x} \delta_x,
        $
        with  
        \begin{equation}\label{eq:atom_mass_in_decomposition_sbb}
        \frac{1}{\beta_x} = \frac{1}{1/a + 1 - 1/a \pi_x(\{0\})} > \frac{1}{1/a + 1} = \beta.
        \end{equation}  
        Thus, each $\nu_x$ is a simple $1/\beta_x$-biased probability with $1/\beta_x > \beta$, which implies that $\nu$ is strongly $\beta$-biased.
            
        Ad 2. Suppose that  $\nu(\{0\}) = 0$ and $\nu$ is strongly $\beta$-biased. Denote $I := \big(0, -\frac{s(\nu)}{a}\big)$ and for simplicity write $\gamma_\beta := aR^\beta_{\#} \nu^L + \alpha \delta_0$. Then $S(\gamma_\beta) = -\frac{s(\nu)}{a}$, and it follows from Lemma~\ref{lem:support^Restriction0} that 
        \begin{equation}\label{nuRirr}
        \nu^R( I) = \tm{\nu^R} = \gamma_\beta(\bar I). 
\end{equation}
        Since $\nu^R \prec \gamma_\beta$, in order to show irreducibility (Definition \ref{def: irreducibility}), we need to prove that, for any $k \in I$, we have $\int (k-y)_+ \, \nu^R(dy) < \int (k-y)_+ \, \gamma_\beta(dy)$.
        Since $\nu$ is strongly $\beta$-biased, the first part of the proof guarantees the existence of a coupling $\pi \in \cpls\left( \frac{\nu^R}{\tm{\nu^R}}, \frac{a R^\beta_{\#}\nu^L + \alpha \delta_0}{\tm{\nu^R}} \right)$ such that $\bar{\pi}_x = x$ and $\pi_x(\{0\}) > 0$ for $\nu^R$-almost all $x$. Then \begin{equation}\label{eq:disintegration_of_nuR_coupling}
            \int(k-y)_+ \, \gamma_\beta(dy) = \int \int(k-y)_+ \, \pi_x(dy) \, \nu^R(dx).
        \end{equation}          
        By Jensen's inequality, for every $k$ and for $\nu^R$-almost all $x$, we have $\int (k-y)_+ \, \pi_x(dy) \geq (k-x)_+$.  Moreover, for $\nu^R$-almost every $x$, since $\pi_x(\{0\}) > 0$ and $\bar \pi_x = x$, we have the strict inequality
        $$
        \int (k-y)_+\,\pi_x(dy) > (k-x)_+
        $$
        for all $k \in (0, S(\pi_x))$.
        Therefore, for $k \in I$, we obtain  
        $$\textstyle \nu^R(\{x: \int(k-y)_+ \, \pi_x(dy) > (k-x)_+ \}) > 0,$$  
        which implies  
        $$
        \int \int(k-y)_+ \, \pi_x(dy) \, \nu^R(dx) > \int (k-x)_+ \, \nu^R(dx).
        $$
        Substituting this into \eqref{eq:disintegration_of_nuR_coupling}, we conclude that, for any $k\in I$,  
        \[\int(k-y)_+ \, \gamma_\beta(dy) > \int (k-x)_+ \, \nu^R(dx).\]  
        Together with \eqref{nuRirr}, this shows that $\nu^R \prec_{sc} \gamma_\beta$. 
        
        Now, to prove the other implication, assume that $\nu^R \suc aR^\beta_{\#} \nu^L + \alpha \delta_0$ with $\alpha > 0$. The idea is to construct a martingale coupling 
        \[\pi \in \cpls\left( \frac{\nu^R}{\tm{\nu^R}}, \frac{a R^\beta_{\#}\nu^L + \alpha \delta_0}{\tm{\nu^R}} \right)\]
        such that $\pi_x(\{ 0 \}) > 0$ for $\nu^R$-almost all $x$. Then, by the first statement of the lemma, this will imply that $\nu$ is strongly $\beta$-biased. 
        For any coupling $\pi \in \cpls\left( \frac{\nu^R}{\tm{\nu^R}}, \frac{a R^\beta_{\#}\nu^L + \alpha \delta_0}{\tm{\nu^R}} \right)$, we have $0 < \alpha = \int \pi_x(\{0\}) \nu^R(dx)$, which implies that $\pi_x(\{0\}) > 0$ for a set of $x$'s with positive $\nu^R$-measure.  
        According to \cite{BaBeScTs23}, the stretched Brownian motion coupling $\pi^{SBM} \in \cpls\left( \frac{\nu^R}{\tm{\nu^R}}, \frac{a R^\beta_{\#}\nu^L + \alpha \delta_0}{\tm{\nu^R}} \right)$ has the property that a.s. $\pi_x^{SBM}$ 
         is equivalent to the second marginal, where $(\pi_x^{SBM})_x$ denotes the disintegration with respect to the first marginal. Since $a R^\beta_{\#}\nu^L + \alpha \delta_0$ assigns positive mass to $\{0\}$, it follows that
        $$
        \pi^{SBM}_x(\{0\}) > 0 \quad \text{for } \nu^R\text{-almost every }x.
        $$
        Taking $\pi = \pi^{SBM}$ and applying the first statement of the lemma, we complete the proof.
\end{proof}

Similarly to Corollary \ref{cor:Characterization_biased_around_x}, the statements of Lemma \ref{L: stronglybetabiased_characterization} extend to probabilities that are strongly $\beta$-biased around $x$, as follows:
\begin{corollary}\label{cor:Characterization_biased_around_x_strongly}
Let $\nu \in \Pc_1(\R)$. Then:
\begin{enumerate}
    \item  $\nu$ is strongly $\beta$-biased around $x$ if and only if there is a coupling  $$\pi \in \cpls\left( \frac{\nu^R}{\tm{\nu^R}}, \frac{a R^{\beta,x}_{\#}\nu^L + \alpha \delta_x}{\tm{\nu^R}} \right),$$ with $\alpha = \tm{\nu^R} - a\tm{\nu^L} > 0$, such that $\bar{\pi}_z = z$ and $\pi_z(\{x\}) > 0$ for $\nu^R$-almost all $z$. 
    \item Suppose that $\nu(\{x\}) = 0$. Then $\nu$ is strongly $\beta$-biased around $x$ if and only if $$\nu^R \suc aR^{\beta,x}_{\#} \nu^L + \alpha \delta_x,$$ with $\alpha = \tm{\nu^R} - a\tm{\nu^L} > 0$.
\end{enumerate}
\end{corollary}

\begin{remark}
The last corollary implies that a symmetric distribution $\nu \in \mathcal{P}_1(\mathbb{R})$ cannot be strongly $1/2$-biased, while it is $1/2$-biased (see Remark~\ref{rem:symm}).
\end{remark}

\subsection{Strong biased order and strongly biased martingale couplings}

Similarly to Definition~\ref{def:bbc-coupling}, we introduce the strong $\beta$-biased order and the set of strongly $\beta$-biased martingale couplings:

\begin{definition}\label{def:sbbc-order}
Let $\mu, \nu \in \Pc_1(\R)$. The set of strongly $\beta$-biased martingale couplings of $\mu$ and $\nu$, denoted by $\cplsbbc(\mu, \nu)$
consists of all couplings $\pi$ of $\mu$ and $\nu$ such that the $\mu$-disintegration $(\pi_x)_x$ of $\pi$ satisfies $\mu$-a.s. that $\pi_x$ is strongly $\beta$-biased around $x$.
We say that $\mu, \nu$ are in strong $\beta$-biased order, denoted by $\mu\sbbc\nu$, if $\cplsbbc(\mu, \nu) \neq \emptyset$. In particular, $\nu$ is strongly $\beta$-biased around $x$ if and only if $\delta_x \sbbc \nu$, in which case $x=\bar \nu$.
\end{definition}

In the remainder of the section, we explore the connections between $\beta$-biased probabilities and strongly $\beta$-biased probabilities.
The first simple lemma shows that the set of strongly $\beta$-biased probabilities is dense in the set of $\beta$-biased probabilities.
\begin{lemma}\label{Closure_of_biased}
    The set of $\beta$-biased probabilities is the closure in $\W_1$ of the set of strongly $\beta$-biased probabilities, i.e.
        $$\{ \nu : \delta_x \bbc \nu \} = \overline{ \{ \nu : \delta_x \sbbc \nu \} }.$$
\end{lemma}

\begin{proof}
    From Definition \ref{def:str_biased}, we immediately get that $ \{ \nu : \delta_x \sbbc \nu \} \subseteq \{ \nu : \delta_x \bbc \nu \}$. From Lemma \ref{lem:betabiased_test_functions}, it follows that the set $\{ \nu : \delta_x \bbc \nu \}$ is convex and closed. Therefore, it is sufficient to show that, for any $\nu \in \{ \nu : \delta_x \bbc \nu \}$, there exists a sequence of measures $\nu_n \in \{ \nu : \delta_x \sbbc \nu \}$ such that $\nu_n \to \nu$. Moreover, from Lemma \ref{lem:atomis_as_a_mixture_of_simple}, it is enough to demonstrate the existence of such a sequence only for measures in $\{ \nu : \delta_x \bbc \nu \}$ which are simple $\beta$-biased, i.e., $\nu = \nu^L + \gamma \delta_M$ with $\gamma \geq \beta$.
    Now, consider a sequence $0 < \alpha_n \nearrow 1$ and define measures $\nu_n := \alpha_n \nu^L + (1-\alpha_n(1-\gamma)) \delta_{M_n}$, where $M_n$ are such that $\bar{\nu}_n = x$. Then all $\nu_n$ are simple $(1-\alpha_n(1-\gamma))$-biased, with $(1-\alpha_n(1-\gamma))> \beta$, so $\nu_n \in \{ \nu : \delta_x \sbbc \nu \}$, and clearly $\nu_n \to \nu$.
\end{proof}

Recall that for $\nu\in \Pc_1(\R)$ we write $
p_{\nu}(k) = \int (k-y)_+ \, \nu(dy)$.
Note that, for two measures $\mu$ and $\nu$ with same total mass and barycenter, we have  
$\mu \prec \nu$ if and only if  $p_{\mu}(k) \leq p_{\nu}(k) $ for all $k \in \supp(\mu).$

The following proposition establishes a connection between the properties of being strongly $\beta$-biased and being $\gamma$-biased with $\gamma > \beta$, for probabilities with support bounded from below.
\begin{proposition}\label{prob:stronglybetabiased_fixed}
    Consider a probability measure $\nu \neq \delta_x$ such that $s(\nu) > -\infty$. Then, $\nu$ is $\gamma$-biased around $x$ for some $\gamma > \beta$ if and only if $\nu$ is strongly $\beta$-biased around $x$ and $S(\nu) < \frac{x}{\beta} - \frac{s(\nu)}{a(\beta)}$.
\end{proposition}

\begin{proof}
    If $\nu$ is $\gamma$-biased for some $\gamma>\beta$, then obviously $\nu$ is strongly $\beta$-biased, and from Lemma~\ref{lem:support^Restriction0} it follows that 
    $$
    S(\nu) \leq \frac{x}{\gamma} - \frac{s(\nu)}{a(\gamma)} = s(\nu) + \frac{x-s(\nu)}{\gamma} < s(\nu) + \frac{x-s(\nu)}{\beta} = \frac{x}{\beta} - \frac{s(\nu)}{a(\beta)}.
    $$
    Now suppose that $\nu$ is strongly $\beta$-biased and $S(\nu) < \frac{x}{\beta} - \frac{s(\nu)}{a(\beta)}$. Without loss of generality, we can assume that $x=0$ and $\nu(\{0\}) = 0$. Indeed, if $\nu(\{0\}) \neq 0$, we can consider the decomposition 
    $\nu = \nu(\{0\})\delta_0 + (1-\nu(\{0\}))\nu^1$, 
    where $\nu^1$ is a probability measure supported outside zero. Then the statement of the lemma holds for $\nu$ if and only if it holds for $\nu^1$. 
    For $\theta \in (0,1)$, we define 
    $\alpha(\theta) = \tm{\nu^R} - a(\theta)\tm{\nu^L}$. 
    The idea is to show that there exists $\hat \beta > \beta$ such that 
    \begin{equation}\label{eq:new_beta_sbc_1}
        \nu^R \prec a(\hat \beta) R^{\hat \beta, 0}_{\#} \nu^L + \alpha(\hat \beta) \delta_0,
    \end{equation}
    where $\alpha(\hat \beta) \geq 0$. By Corollary \ref{cor:Characterization_biased_around_x}, this will imply that $\nu$ is $\hat \beta$-biased. 
    For simplicity, we denote 
    $\gamma_\theta := a(\theta) R^{\theta,0}_{\#} \nu^L + \alpha(\theta) \delta_0$, so that $S(\gamma_\theta ) = s(\nu)(1 - 1/\theta)$. Note that for all $\theta$ such that $\alpha(\theta) \geq 0$, the measure $\gamma_\theta$ is non-negative and satisfies $\tm{\gamma_\theta} = \tm{\nu^R}$ and $\bar{\gamma}_\theta = \bar{\nu}^R$. Therefore, \eqref{eq:new_beta_sbc_1} is equivalent to
    \begin{equation}\label{eq:new_beta_sbc_2}
        p_{\nu^R}(k) \leq p_{\gamma_{\hat \beta}}(k)
    \end{equation}
    for all $k \in [0,S(\nu)]$, for some $\hat \beta > \beta$ such that $\alpha(\hat \beta) \geq 0$.

    Now, fix $\tilde{\beta} > \beta$ such that $\alpha(\tilde{\beta}) > 0$ and $S(\gamma_{\tilde{\beta}}) > S(\nu)$. Such a choice is always possible, since both $\alpha(\theta)$ and $S(\gamma_\theta)$ are continuous and decreasing in $\theta$. Since $\nu^R(\{0\}) = 0$, there exists a unique $\tilde{k} > 0$ such that 
    $\alpha(\tilde \beta) \cdot \tilde{k} = p_{\nu^R}(\tilde{k})$ (see Figure~\ref{fig:determining_k_tilde}).
    Since $p_{\nu^R}$ is convex and satisfies $p_{\nu^R}(0)=0$, and since $\nu(\{0\})=0$, we have that
    $
    p_{\nu^R}(k) < \alpha(\tilde{\beta}) k 
    $
    for all $k \in (0, \tilde{k})$.
From $\nu^R \prec_{sc} \gamma_\beta$, it follows that, for any $k \in [\tilde{k}, S(\nu)]$, 
    $$
    p_{\gamma_\beta}(k) - p_{\nu^R}(k) > 0.
    $$
    Then, by continuity of $(\beta,k) \mapsto (p_{\gamma_\beta} - p_{\nu^R})(k)$ and compactness of $[\tilde k, S(\nu)]$, there exists $\hat{\beta} \in (\beta, \tilde{\beta})$ such that 
    $$
    p_{\gamma_{\hat \beta}} - p_{\nu^R}(k) > 0
    $$
    for all $k \in [\tilde{k}, S(\nu)]$. Since $\hat{\beta} \in (\beta,\tilde{\beta})$, it follows that $\alpha(\hat{\beta}) > \alpha(\tilde{\beta}) > 0$ and $S(\gamma_{\hat{\beta}}) > S(\nu)$ (see Figure \ref{fig:construction_gamma_hat_beta}). Therefore, for $k \in [0, \tilde{k}]$ we have
    $$
    p_{\gamma_{\hat \beta}}(k) \geq \alpha(\hat{\beta}) k \geq \alpha(\tilde{\beta}) k \geq p_{\nu^R}(k).
    $$
    Thus, for all $k \in [0, S(\nu)]$ we have $p_{\nu^R}(k) \leq p_{\gamma_{\hat \beta}}(k)$, implying $\nu^R \prec \gamma_{\hat{\beta}}$.
\end{proof}
\begin{figure}[H]
    \centering
    \begin{subfigure}[b]{0.49\textwidth}
        \centering       \includegraphics[width=\textwidth]{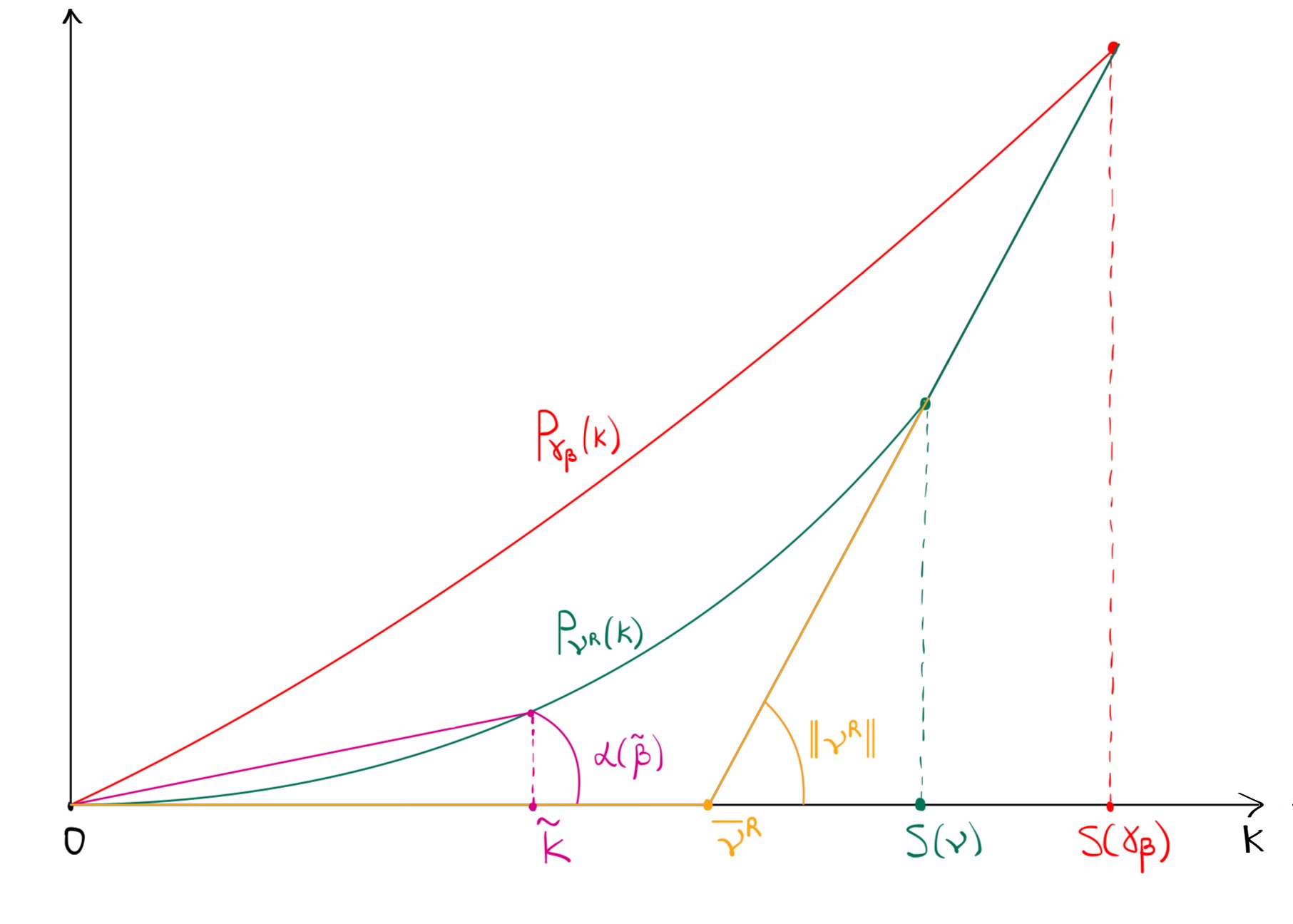}
        \caption{The value $\tilde k$ is defined as the intersection of the green curve $p_{\nu^R}(k)$ and the pink line $\alpha(\tilde \beta) k$. The red curve represents $p_{\gamma_\beta}(k)$, while the yellow curve corresponds to $\tm{\nu^R}(k - \bar{\nu}^R)_+$.}
        \label{fig:determining_k_tilde}
    \end{subfigure}
    \hfill
    \begin{subfigure}[b]{0.49\textwidth}
        \centering        \includegraphics[width=\textwidth]{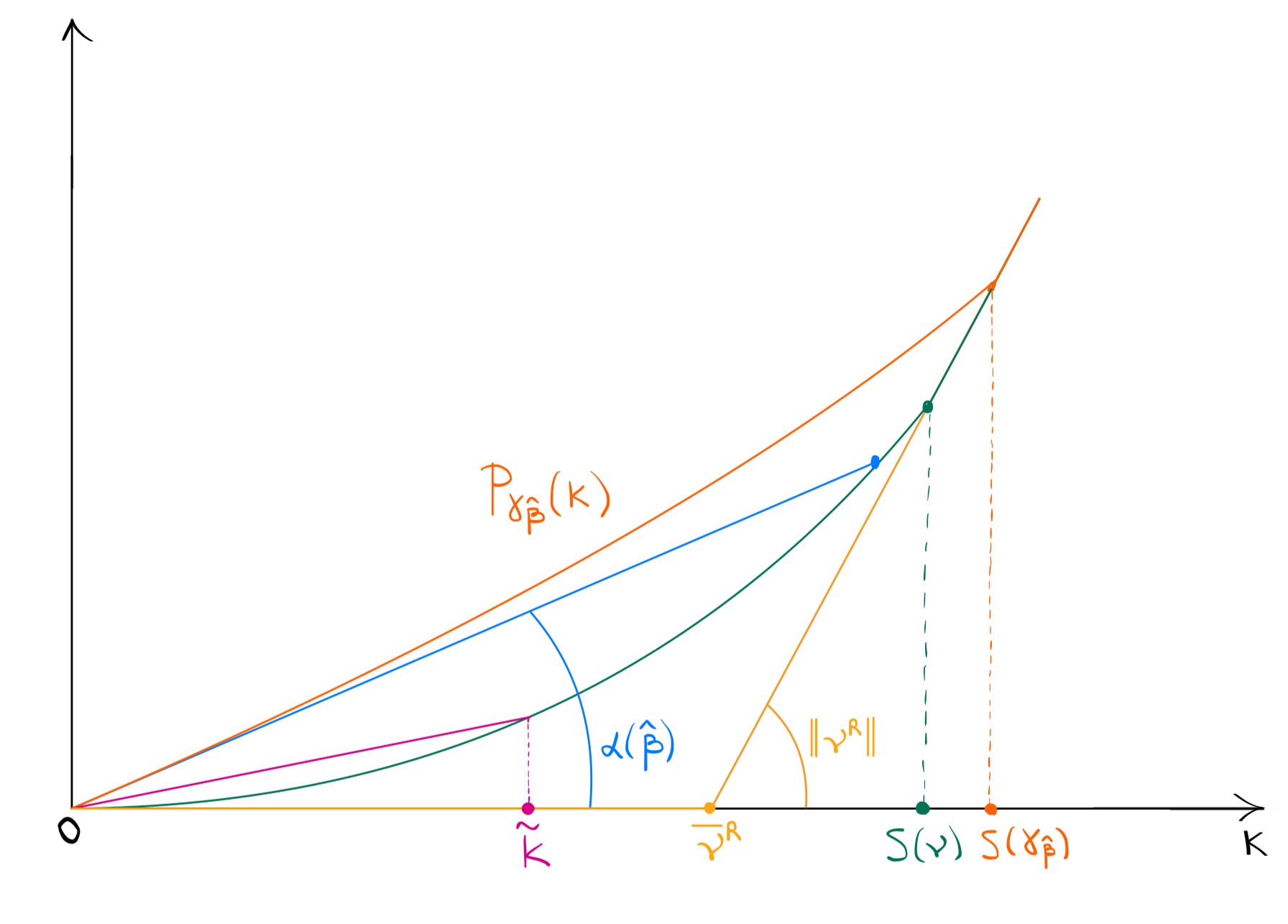}
        \caption{By properly choosing $\hat \beta \in ( \beta, \tilde \beta)$, we ensure that the orange curve $p_{\gamma_{\hat \beta}}(k)$ remains above the green curve $p_{\gamma_\beta}(k)$ for all $k \in [0, S(\nu)]$, while also satisfying $S(\nu) < S(\gamma_{\hat \beta})$.}        \label{fig:construction_gamma_hat_beta}
    \end{subfigure}
    \caption{}
    \label{fig:proof_illustration}
\end{figure}

In the following example we provide  a strongly $\beta$-biased probability measure which fails to be $\gamma$-biased for any $\gamma > \beta$.
\begin{example}
    Let
    $
    \nu = \int_{(0,1)} \left( \tfrac12 u  \delta_{-1} + \left(1 - \tfrac12 u\right)  \delta_{\frac{u}{2 - u}} \right) \, du.
    $
    By construction, $\nu$ is strongly $1/2$-biased. However, since $S(\nu) = -s(\nu) = 1$, Proposition~\ref{prob:stronglybetabiased_fixed} implies that $\nu$ is not $\gamma$-biased for any $\gamma > 1/2$.
\end{example}

\section{\texorpdfstring{$\beta$-biased order and Poisson martingales}{}}
In this section, we characterize the $\beta$-biased order in terms of stochastic integrals. It is well known that for the classical convex order we have $\mu \prec \nu$ if and only if there exist $X_0 \sim \mu$ and a non-negative integrand $H$ such that $X_0 + \int_0^1 H_s\, dB_s \sim \nu$, where $B$ is a Brownian motion. Similarly, we show that the $\beta$-biased order can be characterized through integrals with respect to compensated Poisson processes, which we name Poisson martingales, following \cite{HiYo09}. Specifically, we prove that two distributions are in $\beta$-biased order if and only if they can be coupled via a stochastic integral with respect to the negative compensated Poisson process, with a non-negative integrand, over a time interval of length $\log(1/\beta)$ (Theorem \ref{thm:order_integral_Representation_Intro}). Additionally, we establish a gluing property for $\beta$-biased couplings, which yields a multiplicative version of transitivity (Theorem \ref{thm:transitivity}). Finally, we provide a similar characterization for the strong $\beta$-biased order (Theorem \ref{thm:strict_order_integral_Representation_Intro}).

For the remainder of this section, we fix a standard Poisson process $N$ on a filtered probability space $(\Omega, \mathcal{F}, \mathbb{P}, (\mathcal{F}_t)_{t \geq 0})$, where $\mathcal{F}_0$ is rich enough to support a uniformly distributed random variable. 
As stated in the introduction, we shall provide a characterization of the existence of (strongly) $\beta$-biased martingale couplings in terms of integrals with respect to the negative compensated Poisson process  $M_t= t - N_t$, $t\geq 0$.
For this, an important role will be played by the  first jump time of  $N$, which we denote by $\tau_1$.

\subsection{Gluing properties of $\beta$-biased couplings}
While the $\beta$-biased order is not transitive — that is, $\nu_0 \prec_\beta \nu_1$ and $\nu_1 \prec_\beta \nu_2$ do not imply $\nu_0 \prec_\beta \nu_2$ — the following lemma asserts that $\beta$-biased couplings still exhibit a certain amount of stability under gluing: if $\pi^1 \in \mathrm{Cpl}_{\beta_1}(\nu_0, \nu_1)$ and $\pi^2 \in \mathrm{Cpl}_{\beta_2}(\nu_1, \nu_2)$, then their gluing
$$
\hat{\pi}(dx, dy) = \nu_0(dx) \otimes \left( \int \pi^2_{x_0}(dy)\, \pi^1_x(dx_0) \right)
$$
belongs to $\mathrm{Cpl}_{\beta_1 \beta_2}(\nu_0, \nu_2)$.

\begin{theorem}[$\beta$-biased gluing]\label{thm:transitivity}
Let $\beta_1, \beta_2 \in (0,1)$, and let $\nu_0, \nu_1, \nu_2 \in \Pc_1(\R)$ be such that $\nu_0 \prec_{\beta_1} \nu_1$ and $\nu_1 \prec_{\beta_2} \nu_2$. Then $\nu_0 \prec_{\beta_1 \beta_2} \nu_2$.
Moreover, if at least one of the orders is strict, i.e., $\nu_0 \prec_{s\beta_1} \nu_1$ or $\nu_1 \prec_{s\beta_2} \nu_2$, then  $\nu_0 \prec_{s\beta_1 \beta_2} \nu_2$.
\end{theorem}

\begin{proof} 
        Without loss of generality, let $\nu_0$ be centered at 0.
        Let $\pi \in \text{Cpl}_{\beta_2}(\nu_1, \nu_2)$, and denote its $\nu_1$-disintegration by $(\pi_x)_x$.

        \medskip \emph{Step 1: the statement holds for $\nu_0 = \delta_0$ and $\nu_1$ simple $\beta_1$-biased, if $\pi$ satisfies that $\pi_{x_1}$ is simple $\beta_2$-biased around $x_1$, where $x_1 := S(\nu_1)$.}\\ 
        Let $R_x := S(\pi_x)$. Note that $\nu_1(\{x_1\}) \geq \beta_1$ and $\pi_{x_1}(\{R_{x_1}\}) \geq \beta_2$, hence 
        \begin{equation}\label{eq:lem_transitivity_step1_maxsup}
        \nu_2(\{R_{x_1}\}) \geq \nu_1(\{x_1\}) \,\pi_{x_1}(\{R_{x_1}\})  \geq \beta_1 \beta_2. \end{equation}
        If $R_{x_1} = S(\nu_2)$, then $\nu_2$ is atomic $\beta_1 \beta_2$-biased, and the proof of this step is complete. Otherwise, we construct a representation of $\nu_2$ as a mixture of $\beta_1 \beta_2$-biased probabilities as follows. Consider the set $\Theta := \{x < 0 : R_x > R_{x_1} \}$ (see Figure \ref{fig:transitivity}). 

        \begin{figure}[H]
    \centering
    \includegraphics[width=0.45\textwidth]{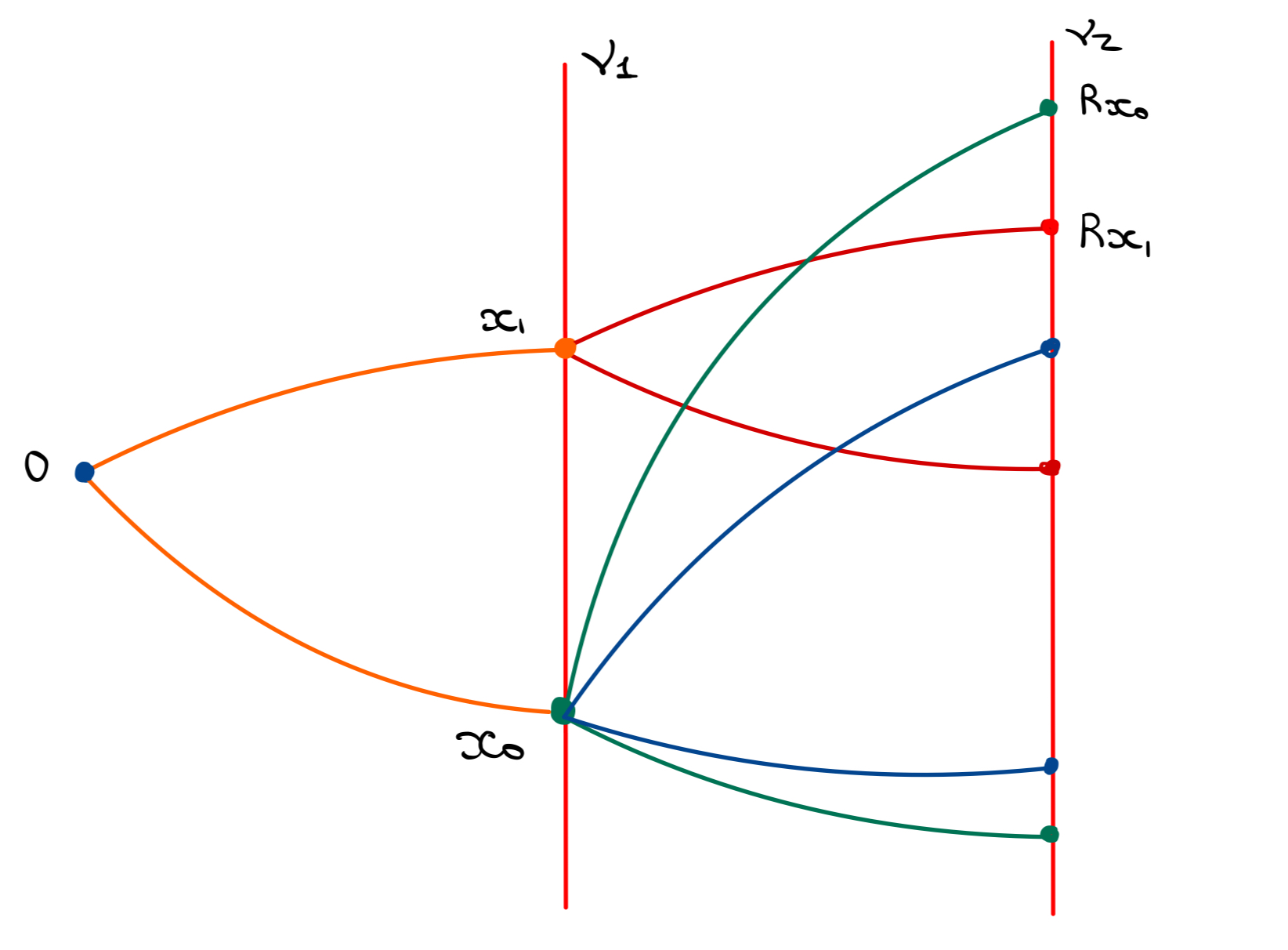}
    \caption{Illustration of Step 1 in the proof of Theorem~\ref{thm:transitivity}, where $\nu_0 = \delta_0$, $\Theta = \{x_0\}$ and $U_{x_0}=\{u_{0}\}$, with corresponding $\pi_{x_0, u_{0}}$ in green.}
   \label{fig:transitivity}
\end{figure}
Since for $\nu_1$-a.e.\ $x$ it holds that $\delta_x \prec_{\beta_2} \pi_x$, it follows from Lemma~\ref{lem:atomis_as_a_mixture_of_simple} and Remark~\ref{rem:meausurability_of_decomposition} that $\pi_x$ admits a representation as a mixture of simple $\beta_2$-biased probabilities centered at $x$, i.e., $
        \pi_x = \int \pi_{x,u} \, \lambda(du)$,
        where each $\pi_{x,u}$ is simple $\beta_2$-biased around $x$, and the map $(x,u) \mapsto \pi_{x,u}$ is measurable. For $x \in \Theta$, the set $U_x := \{ u : S(\pi_{x,u}) > R_{x_1} \}$ has positive Lebesgue measure since $$0 < \pi_x\big((R_{x_1}, +\infty)\big) = \int_{U_x} \pi_{x,u}\big((R_{x_1}, +\infty)\big) \, \lambda(du),$$
        and it holds that $\bar{\pi}_{x,u} = x < 0$. Then it is straightforward to see, for example, by Jankov--von Neumann uniformization theorem, that there exists an analytically measurable function $\varepsilon : \R^2 \to (0,1)$ such that, for $x \in \Theta$ and $u \in U_x$, the measure
        \begin{equation}\label{eq:thm_transitivity_eta_xu}
        \eta_{x,u} := \pi_{x,u}^R + (1 - \varepsilon(x,u))\pi_{x,u}^L
        \end{equation}
        satisfies $\bar{\eta}_{x,u} = 0$.
        Now, for $x \in \Theta$, set 
        $\eta_x := \int_{U_x} \eta_{x,u} \, \lambda(du)$
        and $\gamma := \int_{\Theta} \tm{\eta_x} \, \nu_1(dx)$. Furthermore, define\begin{equation}\label{eq:thm_transitivity_mu_x1}
        \mu_{x_1} := \frac{1}{1-\gamma} \left( \nu_2 - \int_{\Theta} \eta_x \, \nu_1(dx) \right),
        \end{equation}
        and set $\mu_x := \frac{\eta_x}{\tm{\eta_x}}$ for $x \in \Theta$. Note that $\mu_{x_1}$ is non-negative, since $\nu_2 = \int \int \pi_{x,u} \, \lambda(du) \, \nu_1(dx)$ and $\eta_x = \int_{U_x} \left(\pi_{x,u}^R + (1-\varepsilon(x,u))\pi_{x,u}^L\right)\, \lambda(du)$. Since each $\pi_{x,u}^R$ consists of a single atom with mass greater than or equal to $\beta_2$, from \eqref{eq:thm_transitivity_eta_xu} for $x \in \Theta$ it follows that $\frac{\eta_{x,u}}{\tm{\eta_{x,u}}}$ is strongly $\beta_2$-biased, and 
        $$
        \mu_x = \int_{U_x} \frac{\eta_{x,u}}{\tm{\eta_{x,u}}} \, \hat{\rho}_x(du),
        $$ 
        where $\hat{\rho}_x(du) := \frac{\tm{\eta_{x,u}}}{\tm{\eta_x}} \, \lambda(du)$, is a representation of $\mu_x$ as a mixture of strongly $\beta_2$-biased probabilities, hence $\delta_0 \prec_{s\beta_2} \mu_x$. From \eqref{eq:thm_transitivity_mu_x1} it follows that $\mu_{x_1}$ is centered at 0 and $S(\mu_{x_1}) = R_{x_1}$, since $\nu_2$ and $\eta_x$ are centered at 0, and 
        \begin{align*}
        \nu_2\bigl((R_{x_1}, +\infty)\bigr)
        &= \int_\Theta \int_{U_x} 
              \pi_{x,u}\bigl((R_{x_1}, +\infty)\bigr)\,
              \lambda(du)\,\nu_1(dx) \\
        &= \int_\Theta 
              \eta_x\bigl((R_{x_1}, +\infty)\bigr)\,
              \nu_1(dx).
        \end{align*}
        From \eqref{eq:lem_transitivity_step1_maxsup} and \eqref{eq:thm_transitivity_mu_x1} it follows that 
        $$
        \mu_{x_1}(\{R_{x_1}\}) = \frac{\nu_2(\{R_{x_1}\})}{1-\gamma} > \beta_1 \beta_2,
        $$
        implying that $\mu_{x_1}$ is strongly $\beta_1 \beta_2$-biased, and
        \begin{equation}\label{eq:lemma_transitivity_decompisition_to_strict}
                \nu_2 = (1-\gamma) \mu_{x_1} + \gamma \int_\Theta \mu_x \, \rho(dx),
        \end{equation}
        where $\rho(dx) = \frac{\tm{\eta_x} \, \nu_1(dx)}{\gamma}$, is a decomposition of $\nu_2$ into a mixture
        of strongly $\beta_1 \beta_2$-biased probability measures. Therefore, in the case $R_{x_1} < S(\nu_2)$ we even have a strong order $\delta_0 \prec_{s \beta_1 \beta_2} \nu_2$.

        \medskip 
        \emph{Step 2:  the statement holds for $\nu_0 = \delta_0$ and $\nu_1$ simple $\beta_1$-biased.}\\
        Since $\delta_{x_1} \prec_{\beta_2} \pi_{x_1}$, it follows from Lemma~\ref{lem:atomis_as_a_mixture_of_simple} and Remark~\ref{rem:meausurability_of_decomposition} that we can measurably decompose $\pi_{x_1}$ into a mixture of simple $\beta_2$-biased probabilities centered at $x_1$, i.e., $\pi_{x_1} = \int \pi_{x_1, u} \, \lambda(du)$, where each $\pi_{x_1, u}$ is simple $\beta_2$-biased around $x_1$. For $x < x_1$, we set $ \pi_{x,u} := \pi_x$ for any $u$, so that
        $$
        \nu_2 = \int \int  \pi_{x,u} \, \nu_1(dx) \, \lambda(du).
        $$
        For $\nu_{2,u} := \int  \pi_{x,u} \, \nu_1(dx)$, we have $\nu_1 \prec_{\beta_2} \nu_{2,u}$, hence from Step 1 we get $\delta_0 \prec_{\beta_1 \beta_2} \nu_{2,u}$. Since $\nu_2 = \int \nu_{2,u} \, \lambda(du)$, it follows that $\delta_0 \prec_{\beta_1 \beta_2} \nu_2$.
        
        \medskip \emph{Step 3:  the statement holds in the general setting.}
        \\[0.1cm]
        Let $\hat{\pi} \in \text{Cpl}_{\beta_1}(\nu_0, \nu_1)$, and denote its $\nu_0$-disintegration by $(\hat{\pi}_{x_0})_{x_0}$. As above, for a.e. $x_0$, we can measurably decompose $\hat{\pi}_{x_0}$ into a mixture of simple $\beta_1$-biased probabilities, i.e., $\hat{\pi}_{x_0} = \int \hat{\pi}_{x_0, u} \, \lambda(du)$, where each $\hat{\pi}_{x_0, u}$ is simple $\beta_1$-biased around $x_0$, and $(x_0,u)\mapsto \hat \pi_{x_0,u}$ is measurable. Since $\pi \in \text{Cpl}_{\beta_2}(\nu_1, \nu_2)$, we have
        \begin{align*}
        \nu_2
        &= \int \pi_x \,\nu_1(dx)
         = \int \!\int \pi_x \,\hat{\pi}_{x_0}(dx)\,\nu_0(dx_0) \\
        &= \int \!\int \!\int \pi_x \,\hat{\pi}_{x_0,u}(dx)\,\lambda(du)\,\nu_0(dx_0).
        \end{align*}
        Let $\nu_{2, x_0, u} := \int \pi_x \, \hat{\pi}_{x_0, u}(dx)$ and 
        $\nu_{2, x_0} := \int \nu_{2, x_0, u} \, \lambda(du)$. Since $\delta_x \prec_{\beta_2} \pi_x$ and $\hat{\pi}_{x_0, u}$ is simple $\beta_1$-biased around $x_0$, it follows from Step 2 that $\delta_{x_0} \prec_{\beta_1 \beta_2} \nu_{2, x_0, u}$, and thus $\delta_{x_0} \prec_{\beta_1 \beta_2} \nu_{2, x_0}$. Therefore, since $\nu_2 = \int \nu_{2, x_0} \, \nu_0(dx_0)$, we conclude that $\nu_0 \prec_{\beta_1 \beta_2} \nu_2$.

        The proof of the assertion for the strict order follows by repeating the above proof, with the only modification being that, in Step 1, the second inequality in \eqref{eq:lem_transitivity_step1_maxsup} becomes strict.
\end{proof}

The following lemma provides two useful sufficient conditions for a probability measure to be strongly $\beta$-biased. 
\begin{lemma}\label{lem:sufficient_conditions_strict_order}
    Let $\beta_1, \beta_2 \in (0,1)$ and let $\nu_1, \nu_2 \in \mathcal{P}_1(\mathbb{R})$, with $\nu_1$ atomic $\beta_1$-biased. Suppose that $\nu_1 \prec_{\beta_2} \nu_2$, and that there exists $\pi \in \mathrm{Cpl}_{\beta_2}(\nu_1, \nu_2)$ such that $\pi_{x_1}$ is simple $\beta_2$-biased around $x_1 := S(\nu_1)$. If one of the following holds:
    \begin{enumerate}
        \item $S(\pi_{x_1}) < S(\nu_2)$, 
        \item $\pi_{x_1}\big(\{S(\pi_{x_1})\}\big) > \beta_2$,
    \end{enumerate}
    then $\delta_0 \prec_{s\beta_1 \beta_2} \nu_2$.
\end{lemma}

\begin{proof}
    Suppose first that $S(\pi_{x_1}) < S(\nu_2)$ (see Figure \ref{fig:nu2_decomposition}). If $\tm{\nu_1^R} > \beta_1$, then $\delta_0 \prec_{s \beta_1} \nu_1$ by Lemma~\ref{lem:atomic_not_simple_implies_stronglybiased}, and Theorem \ref{thm:transitivity} implies that $\delta_0 \prec_{s \beta_1 \beta_2} \nu_2$. Otherwise, if $\nu_1$ is a simple $\beta_1$-biased probability, the result follows directly from Step 1 of the proof of Theorem \ref{thm:transitivity}, as  \eqref{eq:lemma_transitivity_decompisition_to_strict} provides a decomposition of $\nu_2$ into a mixture of strongly $\beta_1 \beta_2$-biased probabilities. 

    Next, suppose that $\pi_{x_1}\big(\{S(\pi_{x_1})\}\big) > \beta_2$. Then either $S(\pi_{x_1}) < S(\nu_2)$ or $S(\pi_{x_1}) = S(\nu_2)$. In the first case, the result follows from part 1 of the proof. In the second case, $\nu_2\big(\{S(\nu_2)\}\big) > \beta_1 \beta_2$, which directly implies that $\delta_0 \prec_{s \beta_1 \beta_2} \nu_2$.
\end{proof}

\begin{figure}[H]
    \centering
    \includegraphics[width=0.45\textwidth]{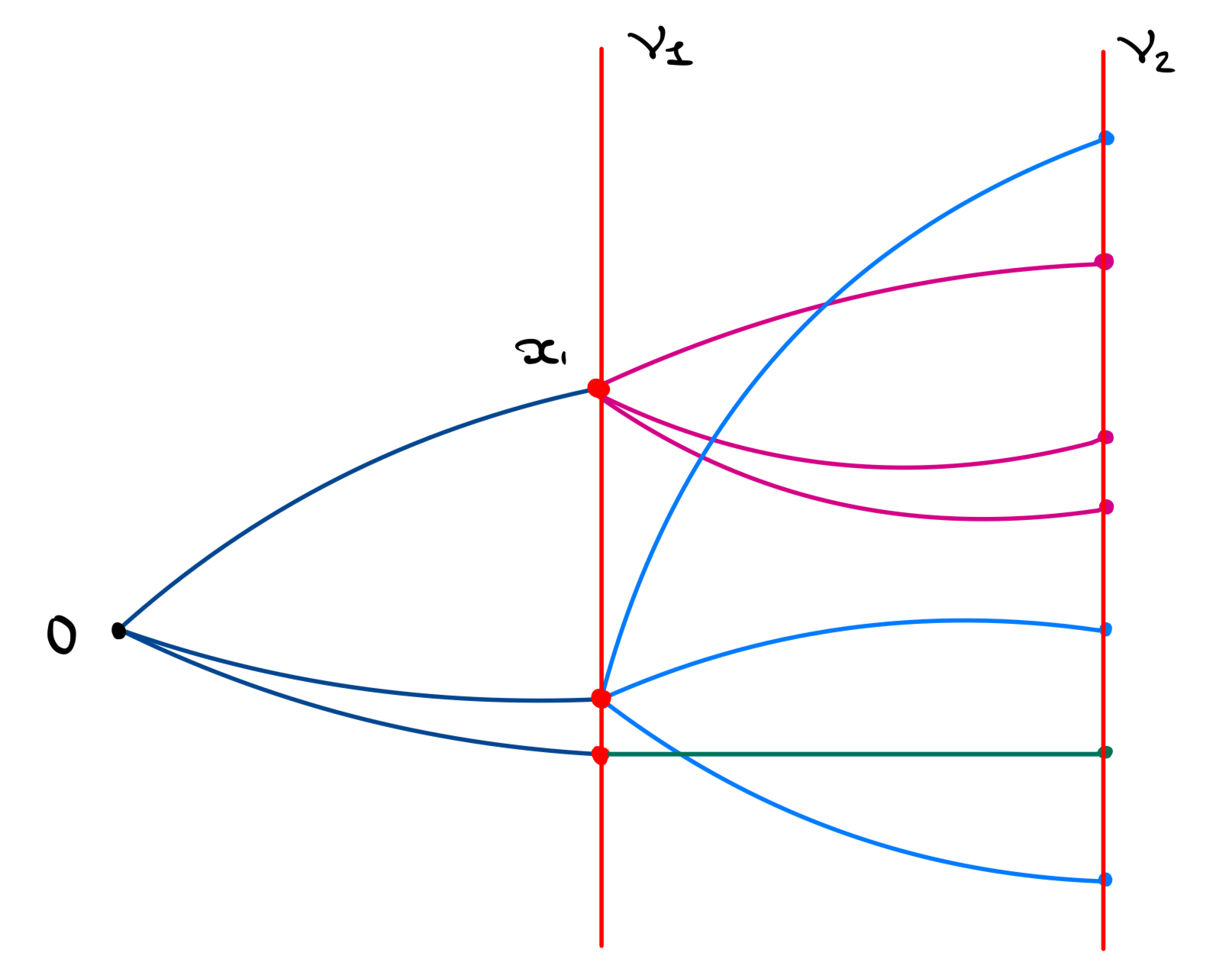} 
    \caption{$\max \supp(\pi_{x_1}) < \max \supp(\nu_2)$ implies $\delta_0 \prec_{s\beta_1 \beta_2} \nu_2$.}
    \label{fig:nu2_decomposition}
\end{figure}

\subsection{Integral characterization of $\beta$-biased order}
\begin{lemma}\label{lem:peackock_simple_integrand}
Let $X, H\in L^1(\F_0)$, $H\geq0$ and $Y=X+ H M_t=X+\int_0^t H \, dM_s$ for some $t>0$. Then 
$$ \law(X) \prec_{\beta_t} \law(Y), \quad \text{where } \beta_t = e^{-t}. $$ 
\end{lemma}

\begin{proof}
    Denote by $\mu$ the distribution of $X$, by $\rho$ the joint distribution of $(X, H)$, and let $(\rho_x)_x$ be the $\mu$-disintegration of $\rho$. Since $M_t$ is independent of $\mathcal{F}_0$, we have
    $$
    \law(Y) = \int \int \law(x + h M_t) \, \rho_x(dh) \, \mu(dx),
    $$
    hence the assertion holds if we show that $\delta_x \prec_{\beta_t} \law(x + h M_t)$ for $\rho$-a.e. $(x,h)$.
    Now, for any $h \geq 0$, the random variable $x + h M_t = x + h t - h N_t$ has mean $x$ and attains its maximum value $x + h t$ when the Poisson process $N$ does not jump on the interval $(0, t]$. Thus, $\P[x + h M_t = x + h t] = e^{-t}$, which implies that $\law(x + h M_t)$ is atomic $\beta_t$-biased around $x$, with $\beta_t = e^{-t}$.
\end{proof}

\begin{corollary}\label{corollary: biased-peacock}
    Let $X_0 \in L^1(\F_0)$, and $H\geq0$ be a predictable process satisfying $\E\big[\int_0^T H_s ds\big]<\infty$. For every $t\in[0,T]$, let $X_t := X_0 + \int_0^t H_s \, dM_s$ and denote by $\nu_t$ its distribution. Then, for any $0\leq s < t \leq T$, we have 
    $$\nu_s \prec_{\beta_{t-s}} \nu_t,\qquad \text{where\; $\beta_{t-s} = e^{-(t-s)}$}.$$
\end{corollary}

\begin{proof}
    If $H$ is simple, i.e. $H_s = \sum_{i=1}^n H_{t_{i}} \mathbf{1}_{t_{i-1} < s \leq t_i}$, with $H_{t_i} \in L^1(\F_{t_{i-1}})$ non-negative and $0 \leq t_0 < \dots < t_n \leq T$, then from Lemma \ref{lem:peackock_simple_integrand} and Theorem \ref{thm:transitivity}, it follows that, for $s < t$, 
    $$
    \nu_{s} \prec_{\beta_{t-s}} \nu_{t}.
    $$
    Then, for any $H \geq 0$ such that $\E\left[\int_0^T H_s \, ds\right] < \infty$, we take a sequence of simple integrands $(H^n)_{n\in\N}$ such that $\E\left[\int_0^T |H_s - H_s^n| \, ds\right] \to 0$ as $n\to\infty$. Denote by $\nu^n_t$ the distribution of $X_0 + \int_0^t H_s^n \, dM_s$. Since 
    \begin{align*}
    \E\!\left[\Bigl|\int_0^T (H_s - H_s^n)\, dN_s\Bigr|\right]
    &\le \E\!\left[\int_0^T |H_s - H_s^n| \, dN_s\right] \\
    &= \E\!\left[\int_0^T |H_s - H_s^n| \, ds\right] \;\to 0, \quad n\to\infty,
    \end{align*}
    we obtain that $\nu^n_t \to \nu_t$ for every $t \in [0,T]$ as $n\to\infty$. From the first step of the proof, we know that $\nu^n_s \prec_{\beta_{t-s}} \nu^n_t$, and therefore for any $g \in C_L(\R)$ we have $\nu^n_s(g_{\beta_{t-s}}) \leq \nu^n_t(g)$. Since $\nu^n_t \to \nu_t$ and $\nu^n_s \to \nu_s$ in $\W_1$, it follows that for any $g \in C_L(\R)$ we have $\nu_s(g_{\beta_{t-s}}) \leq \nu_t(g)$, and thus $\nu_s \prec_{\beta_{t-s}} \nu_t$.  
\end{proof}

We are now in the position to prove our second main result (Theorem~\ref{thm:order_integral_Representation_Intro}), which provides a characterization of the $\beta$-biased order in terms of Poisson martingales. Here and throughout the rest of the paper, for $\beta \in (0,1)$, we denote $t_\beta := \log(1/\beta)$.

\begin{proof}[Proof of Theorem~\ref{thm:order_integral_Representation_Intro}]
    If \eqref{eq:thm_order^Representation} holds, then as an immediate consequence of Corollary~\ref{corollary: biased-peacock} we get that $\mu \bbc \nu$, and so $\cplbbc(\mu,\nu) \neq \emptyset$.    
    We now prove the converse implication. Without loss of generality, we assume that $\bar \mu = 0$. 
    Let $\mu \bbc \nu$ and consider a random variable $X_0 \in L^0(\F_0)$ such that $X_0 \sim \mu$. Our goal is to explicitly construct a non-negative process $H \geq 0$ such that the process $X$ defined in \eqref{eq:thm_order^Representation} satisfies $X_{t_\beta} \sim \nu$. 
    
    \medskip \emph{Step 1: Construction of $H$ for $\mu = \delta_0$ and a simple $\beta$-biased $\nu$.}\\
    First, we assume that $\nu^R$ consists of a single atom of mass $\beta$ at  $S(\nu)$, meaning $\nu = \nu^L + \beta \delta_{S(\nu)}$. We aim at finding a process $H$ of the form
    $$
        H_s = H(s) \, \mathbf{1}_{s \leq \tau_1},
    $$
    for some non-negative deterministic function $H(s)$, where $\tau_1 \sim \text{Exp}(1)$ is the  first jump time of the Poisson process $N$.
   This means that, before  $\tau_1$, $H_s$ follows the deterministic function $H(s)$, and it becomes zero afterwards. Then $X_{t_\beta}$ a.s. takes the form
    $$
    X_{t_\beta}  = 
    \begin{cases}
        \int_0^{t_\beta} H(s) \, ds, & N_{t_\beta} = 0,\\
        \int_0^{\tau_1} H(s) \, ds - H(\tau_1), & N_{t_\beta} > 0,
    \end{cases}
    $$
    or equivalently,
    $$
    X_{t_\beta}  = 
    \begin{cases}
        \int_0^{t_\beta} H(s) \, ds, & \tau_1 > t_\beta,\\
        \int_0^{\tau_1} H(s) \, ds - H(\tau_1), & \tau_1 \leq t_\beta.
    \end{cases}
    $$
    Since $\P[N_{t_\beta} = 0] = e^{-t_\beta} = \beta$ and $H \geq 0$, it follows that $X_{t_\beta} \sim \nu = \nu^L + \beta \delta_{S(\nu)}$ if and only if the following conditions hold:
\begin{equation}\label{eq:conditions_on_determ_H}
        \begin{split}
            \int_0^{t_\beta} H(s) \, ds &= S(\nu), \\ 
            \text{Law}\left(\int_0^{\tau_1} H(s) \, ds - H(\tau_1)\bigg| \, \tau_1 \leq t_\beta \right) &=  \frac{\nu^L}{1-\beta}.
        \end{split}
    \end{equation}
    Moreover, the first equation in \eqref{eq:conditions_on_determ_H} is redundant: if the second equation in \eqref{eq:conditions_on_determ_H} holds, then the first one follows from
    $$
    \beta S(\nu) + (1-\beta) \bar{\nu}^L = 0 = \E[X_{t_\beta}] = \beta \int_0^{t_\beta} H(s) \, ds + (1-\beta) \bar{\nu}^L.
    $$
     Therefore, $X_{t_\beta} \sim \nu = \nu^L + \beta \delta_{S(\nu)}$ holds if and only if
    \begin{equation}\label{eq:condition_on_determ_H}
        \text{Law}\left(\int_0^{\tau_1} H(s) \, ds - H(\tau_1) \bigg| \, \tau_1 \leq t_\beta \right) = \frac{\nu^L}{1-\beta}.
    \end{equation}
    Our goal now is to find a function $G$ such that $\text{Law}\big(G(\tau_1) |\, \tau_1 \leq t_\beta \big) = \frac{\nu^L}{1-\beta}$, and then solve the integral equation 
    $$
        \int_0^{t} H(s) \, ds - H(t) = G(t), \quad t \leq t_\beta,
    $$
    which, by virtue of \eqref{eq:condition_on_determ_H}, will in turn imply that $X_{t_\beta} \sim \nu$. Since $\tau_1 \sim \text{Exp}(1)$, the random variable $1 - e^{-\tau_1}$ is uniformly distributed on $[0,1]$, and $\text{Law}(1 - e^{-\tau_1} | \, \tau_1 \leq t_\beta) = \frac{\lambda|_{[0,1-\beta]}}{1-\beta}$. Let $Q_{\nu^L}$ denote the quantile function of $\nu^L$, defined as $Q_{\nu^L}(p) = \inf\{t: \nu^L((-\infty, t]) \geq p\}$. Then we have $\text{Law}\big(Q_{\nu^L}(1 - e^{-\tau_1}) |\, \tau_1 \leq t_\beta\big) = \frac{\nu^L}{1-\beta}$. Thus, if we find $H \geq 0$ satisfying 
\begin{equation}\label{eq:volterra_equation_on_H}
        \int_0^{t} H(s) \, ds - H(t) = Q_{\nu^L}\big(1 - e^{-t}\big), \quad t \leq t_\beta,
    \end{equation}
    then \eqref{eq:condition_on_determ_H} holds true, and hence $X_{t_\beta} \sim \nu$. Equation \eqref{eq:volterra_equation_on_H} is a linear Volterra equation of the second kind with a constant kernel, which is known to have the unique solution
    \begin{equation}\label{eq:explici_formula_H}
        H(t) = - Q_{\nu^L}(1 - e^{-t}) - \int_0^t e^{t-s} Q_{\nu^L}(1 - e^{-s}) \, ds.     
    \end{equation} 
    Thus, for the case when $\nu^R$ consists of a single atom of mass $\beta$, we have explicitly found an integrand $H \geq 0$ such that $X_{t_\beta} = \int_0^{t_\beta} H_s \, ds \sim \nu$.   
    Similarly, when $\nu^R$ consists of a single atom of mass $\gamma > \beta$, we use the same computations as above, but setting $H_t = 0$ for $t \in \big(\log(1/\gamma), t_\beta\big]$, i.e.
    $$
        H_t = H(s) \, \mathbf{1}_{s \leq \log(1/\gamma) \wedge \tau_1 },
    $$
    where the deterministic function $H(t)$ is given by \eqref{eq:explici_formula_H}. This provides the complete solution for the case of $\mu=\delta_0$ and a simple $\beta$-biased probability $\nu$.
       \begin{figure}[h]
    \centering
    \begin{subfigure}[t]{0.49\textwidth}
        \centering
        \includegraphics[width=\linewidth]{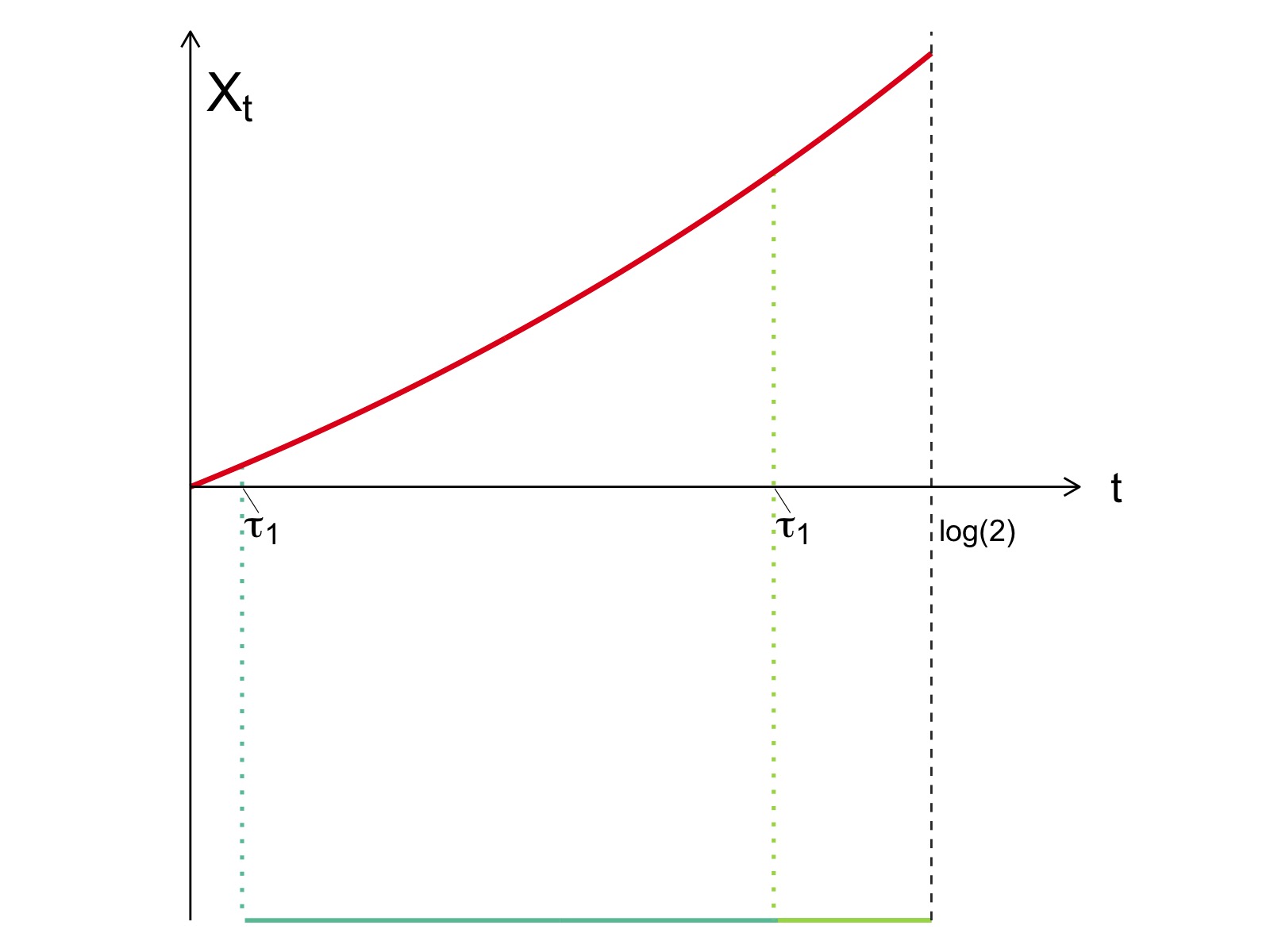}
        \caption{Trajectories of $X_t = \int_0^t H_s \, dM_s$ embedding \\  $\nu = \frac{1}{2} \delta_{-1} + \frac{1}{2} \delta_{1}$.}
        \label{fig:traj_3points}
    \end{subfigure}
    \hfill
    \begin{subfigure}[t]{0.49\textwidth}
        \centering
        \includegraphics[width=\linewidth]{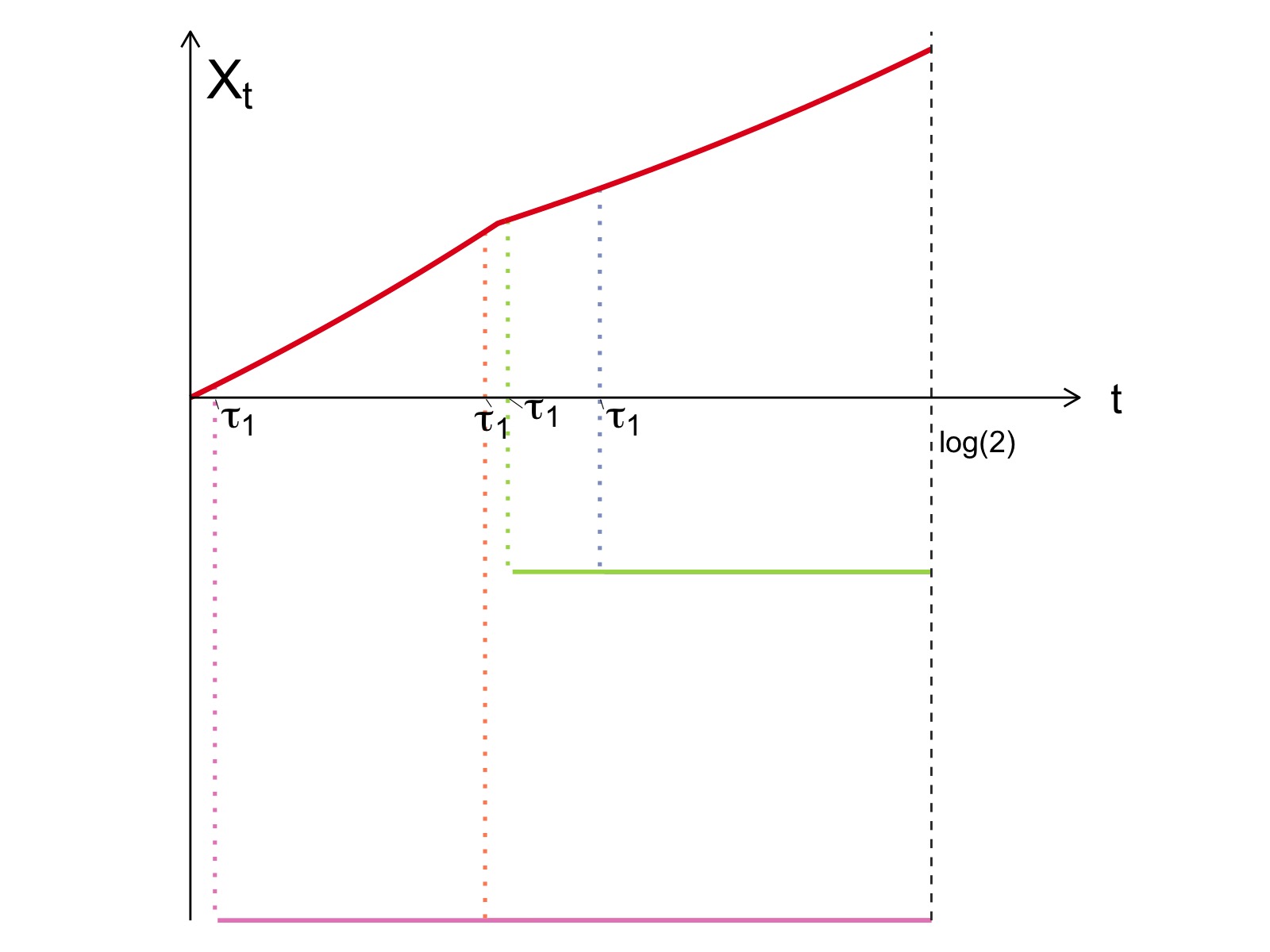}
        \caption{Trajectories of $X_t = \int_0^t H_s \, dM_s$ embedding \\ $\nu = \frac{1}{2} \delta_{2} + \frac{1}{4} \delta_{-1} + \frac{1}{4} \delta_{-3} $.}
        \label{fig:traj_2points}
    \end{subfigure}
    \caption{}
    \label{fig:Thm_1_2_2}
\end{figure}

    \medskip \emph{Step 2: Construction of $H$ for $\mu = \delta_0$ and a general $\beta$-biased $\nu$.}\\
    When $\nu$ is $\beta$-biased, by Lemma \ref{lem:atomis_as_a_mixture_of_simple} and Remark \ref{rem:meausurability_of_decomposition} we can represent $\nu$ as an average of simple $\beta$-biased probabilities  $(\nu_u)_u$, i.e., $\nu = \int \nu_u \, \lambda(du)$. By applying Step 1 to each $\nu_u$, we define $H^u$ such that $\int_0^{t_\beta} H_s^u \, dM_s \sim \nu_u$. It follows from \eqref{eq:explici_formula_H} and measurability of $u \mapsto \nu_u$ that the map $u \mapsto H^u$ is measurable. We then define an $\F_0$-measurable random variable $\xi$ with uniform distribution on $[0,1]$, and set $H := H^\xi$. Then the distribution of $X_{t_\beta}$ satisfies
    \begin{align*}
    \P[X_{t_\beta} \le z]
    &= \int \P\!\left[\int_0^{t_\beta} H_s^u \, dM_s \le z \right] \lambda(du) \\
    &= \int \nu_u\bigl((-\infty, z]\bigr)\, \lambda(du) \\
    &= \nu\bigl((-\infty, z]\bigr),
    \end{align*}
    so that $X_{t_\beta} \sim \nu$.

    \medskip \emph{Step 3: Construction of $H$ in the general setting.}\\
    Consider a coupling $\pi \in \cplbbc(\mu, \nu)$ and its $\mu$-disintegration $(\pi_x)_x$. Since $\delta_x \bbc \pi_x$, by applying  Step 2 to $\pi_x$, we can construct an integrand $H^x$ (which can be taken to be measurable by \eqref{eq:explici_formula_H} and Remark \ref{rem:meausurability_of_decomposition}) such that 
    $$ x + \int_0^{t_\beta} H^x_s \, dM_s \sim \pi_x. $$
    We then define the integrand $H := H^{X_0}$. Since $X_0$ is independent of the Poisson process $N$, we obtain
    \begin{align*}
    \P[X_{t_\beta} \le z]
    &= \int \P\!\left[\,x + \int_0^{t_\beta} H_s^x \, dM_s \le z \right] \mu(dx) \\
    &= \int \pi_x\bigl((-\infty, z]\bigr)\, \mu(dx) \\
    &= \nu\bigl((-\infty, z]\bigr),
    \end{align*}
    so that $X_{t_\beta} \sim \nu$. Thus, we have shown that if $\mu \bbc \nu$, then there exists a martingale $X$ such that $X_0 \sim \mu$, $X_{t_\beta} \sim \nu$, and $X_{t_\beta} = X_0 + \int_0^{t_\beta} H_s \, dM_s$ for some $H \geq 0$.
    \end{proof}

As established in the proof of Theorem~\ref{thm:order_integral_Representation_Intro},
when coupling probability distributions $\mu \prec_\beta \nu$ via a stochastic integral as in \eqref{eq:thm_order^Representation}, it suffices to consider integrands $H$ that become identically zero after $\tau_1$, the first jump time of the Poisson process $N$. 
Specifically, the following was shown.
\begin{corollary}\label{cor:integral_representation_simple_integrand}
    A pair of probability measures $\mu, \nu \in \Pc_1(\R)$ satisfies $\mu \bbc \nu$ if and only if there exists $X_0 \sim \mu$ and a predictable nonnegative process $(H_s)_{s \in [0, t_\beta]}$ of the form $H_s = \bar H(s, \bar{U}, X_0) \mathbf{1}_{s \leq \tau_1}$, where $\bar H$ is deterministic, $\bar{U} \in L^0(\F_0)$, and $\E\left[\int_0^{t_\beta} H_s \, ds \right] < \infty$, such that
    $$
    X_0 + \int_0^{t_\beta} H_s \, dM_s \sim \nu.
    $$ 
\end{corollary}

The following corollary characterizes the specific cases of simple $\beta$-biased and atomic $\beta$-biased probability measures via stochastic integral representations.
\begin{corollary}
\begin{enumerate}
    \item A probability measure $\nu \in \Pc_1(\R)$ with $\nu \ne \delta_x$ is simple $\beta$-biased around $x$ if and only if there exists a nonnegative function $H: [0, t_\beta] \to \R$ such that $\int_0^t H(s)\, ds - H(t) < 0$ for all $t \in (0, t_\beta]$, and
    $$
    x + \int_0^{t_\beta} H(s)\, \mathbf{1}_{s \leq \tau_1} \, dM_s \sim \nu.
    $$

    \item A probability measure $\nu \in \Pc_1(\R)$ is atomic $\beta$-biased around $x$ if and only if there exists a deterministic nonnegative function $H: [0, t_\beta] \to \R$ such that
    $$
    x + \int_0^{t_\beta} H(s)\, \mathbf{1}_{s \leq \tau_1} \, dM_s \sim \nu.
    $$
\end{enumerate}
\end{corollary}

\begin{proof}
To see the first claim, suppose that $\nu$ is simple $\beta$-biased around $x$. The corresponding integrand is given in \eqref{eq:explici_formula_H}. Conversely, given such a function $H$, Theorem~\ref{thm:order_integral_Representation_Intro} yields $\delta_x \prec_\beta \nu$, and

$$
\int_0^{t_\beta} H(s)\, \mathbf{1}_{s \leq \tau_1} \, dM_s =
\begin{cases}
\int_0^{t_\beta} H(s)\, ds, & \tau_1 > t_\beta, \\
\int_0^{\tau_1} H(s)\, ds - H(\tau_1), & \tau_1 \leq t_\beta,
\end{cases}
$$
so that $\nu^R$ consists of a single atom at $x + \int_0^{t_\beta} H(s)\, ds$, with mass at least $\P[\tau_1 > t_\beta] = \beta$.%

To see the second claim, suppose that $\nu$ is atomic $\beta$-biased around $x$ with $\nu(\{S(\nu)\}) = \gamma \geq \beta$, and define
\[
Q(p) := \inf\left\{t: \nu\big((-\infty, t] \cap (-\infty, S(\nu))\big) \geq p\right\}.
\]
Set
$$
H(t) := 
\begin{cases}
-Q(1 - e^{-t}) - \int_0^t e^{t-s} Q(1 - e^{-s})\, ds, & t \in [0, \log(1/\gamma)], \\
0, & t > \log(1/\gamma).
\end{cases}
$$
Then, as in the proof of Theorem~\ref{thm:order_integral_Representation_Intro},
$$
\text{Law}\left(x + \int_0^{\tau_1} H(s)\, ds - H(\tau_1) \Big| \tau_1 \leq \log(1/\gamma) \right) = \frac{\nu|_{(-\infty, S(\nu))}}{1-\gamma},
$$
and therefore $x + \int_0^{t_\beta} H(s)\, \mathbf{1}_{s \leq \tau_1} \, dM_s \sim \nu$.

Conversely, if $H \geq 0$ is deterministic, then
$
x + \int_0^{t_\beta} H(s)\, \mathbf{1}_{s \leq \tau_1} \, dM_s \leq x + \int_0^{t_\beta} H(s)\, ds,
$
and
$$
\P\left[x + \int_0^{t_\beta} H(s)\, \mathbf{1}_{s \leq \tau_1} \, dM_s = x + \int_0^{t_\beta} H(s)\, ds\right] \geq \beta.
$$
Hence, $S(\nu) = x + \int_0^{t_\beta} H(s)\, ds$ and $\nu(\{S(\nu)\}) \geq \beta$, so that $\nu$ is atomic $\beta$-biased around $x$.
\end{proof}

We conclude this section with a corollary that will be used in the proof of Theorem~\ref{thm:strict_order_integral_Representation_Intro}.

\begin{proposition}\label{prop:extension_of_integrand}
    Let $X_t = \int_0^t H_s \mathbf{1}_{s\leq\tau} \, dM_s$ for $t \in [0, T]$, where $H \geq 0$ is predictable and satisfies 
    $\mathbb{E}\left[\int_0^T H_s \, ds\right] < \infty$, and $\tau$ is a stopping time. 
    Then, for any $t_0 < T$, there exists an extension $(\widetilde{\Omega}, \widetilde{\F}, \widetilde{\P})$ of the probability space, equipped with a filtration $\widetilde{\mathbb{F}}=(\widetilde{\F}_s)_{s \in [0, T]}$, on which we can define processes $(\widetilde{H}_s)_{s \in [0, T]}$, $(\widetilde{M}_s)_{s \in [0, T]}$, and a stopping time $\widetilde{\tau}$, such that:
    \begin{enumerate}
        \item $\widetilde{M}$ is a $(\widetilde{\P},\widetilde{\F})$-negative compensated Poisson process,
        \item $\widetilde{H} \geq 0$ is $\widetilde{\mathbb{F}}$-predictable 
        and satisfies $\mathbb{E}_{\widetilde{\P}}\left[\int_0^T \widetilde{H}_s \, ds\right] < \infty$,
        \item For $s \in [0, t_0]$, it holds that $\widetilde{H}_s \mathbf{1}_{s \leq \widetilde{\tau}} = \bar{H}(s, \bar{U}) \mathbf{1}_{s \leq \widetilde{\tau}_1}$ for some deterministic function $\bar{H}$, $\bar{U} \in L^0(\widetilde{\mathcal{F}}_0)$, and where $\widetilde{\tau}_1$ is the first jump time of $\widetilde{M}$,
        \item For $Y_t := \int_0^t \widetilde{H}_s \mathbf{1}_{s \leq \widetilde{\tau}}\, d\widetilde{M}_s$, it holds that
    $Y_{t_0} \overset{d}{=} X_{t_0}$, and for $\law(X_{t_0})$-a.e. $x$,
    \begin{align*}
    \law_{\,\P}\!\big((H_s, M_s - M_{t_0})_{s\in[t_0,T]},\; \tau \vee t_0 \mid X_{t_0}=x\big) \\
    = \law_{\,\widetilde{\P}}\!\big((\widetilde{H}_s, \widetilde{M}_s - \widetilde{M}_{t_0})_{s\in[t_0,T]},\;
       \widetilde{\tau} \vee t_0 \mid Y_{t_0}=x\big).
    \end{align*}
    \end{enumerate}
\end{proposition}

\begin{proof}
    We explicitly construct the required processes and the extension of the probability space. By Corollary \ref{cor:integral_representation_simple_integrand}, applied to $\mu = \delta_0$, $\nu = \law(X_{t_0})$ and $\beta=e^{-t_0}$, there exist a deterministic function $\bar{H}$ and a random variable $\bar{U} \in L^0(\mathcal{F}_0)$ such that 
    $X_{t_0} \overset{d}{=} \int_0^{t_0} \bar{H}(s, \bar{U}) \mathbf{1}_{s \leq \tau_1} \, dM_s =: \xi$. Let $\hat{\mathbb{T}}$ denote the space of functions 
$(a, m)$ on $[t_0, T]$, where $a$ is increasing and absolutely continuous and $m$ is c\`adl\`ag. Define $\mathbb{T} := \hat{\mathbb{T}} \times [t_0, T]$. Let $\gamma$ be a probability kernel from $\R$ to $\mathbb{T}$, i.e., a measurable map $\R \to  \Pc(\mathbb{T})$, defined by
    \[
        \gamma(x; d(a,m,t)) := \law\left(\left(\int_{t_0}^s H_u \, du,\ M_s - M_{t_0}\right)_{s \in [t_0, T]},\ \tau \vee t_0\ \middle|\ X_{t_0} = x\right).
    \]
    On $\hat{\mathbb{T}}$ introduce the coordinate processes $a_s(a, m) := a_s$ and $m_s(a, m) := m_s$. Define the extended space as $\widetilde{\Omega} := \Omega \times \mathbb{T}$, with the filtration $\widetilde{\mathbb{F}} = (\widetilde{\F}_s)_{s \in [0,T]}$ given by
   $$
    \widetilde{\F}_s := 
    \begin{cases} 
        \mathcal{F}_s \otimes \{\emptyset, \hat{\mathbb{T}}\} \otimes \{\emptyset, [t_0, T]\}, & s \leq t_0,\\
        \mathcal{F}_{t_0} \otimes \sigma\left((a_u)_{t_0 \leq u \leq s}, (m_u)_{t_0 \leq u \leq s}\right) \otimes \sigma(\{ [t_0,u] : t_0 \le u \le s \}), & s > t_0.
    \end{cases}
    $$
    Set $\widetilde{\F} := \widetilde{\F}_T$, and define the probability measure $\widetilde{\mathbb{P}}$ on $(\widetilde{\Omega}, \widetilde{\F})$ by
    $$
    \widetilde{\mathbb{P}}(d(\omega, a, m, t)) := \mathbb{P}(d\omega) \otimes \gamma\left(\xi(\omega); d(a, m, t)\right).
    $$
     For $s \in (t_0,T]$ define $h_s := \lim \sup_{\Delta \downarrow 0} \frac{a(s) - a(s-\Delta)}{\Delta}$. Finally, define the processes $\widetilde{M}$ and $\widetilde{H}$ by
    \begin{equation}\label{eq:Extended_integrand}
        \widetilde{M}_s := 
        \begin{cases}
            M_s & s \leq t_0, \\ 
            M_{t_0} + m_s & s > t_0,
        \end{cases} \quad
        \widetilde{H}_s := 
        \begin{cases}
            \bar{H}(s, \bar{U}) & s \leq t_0, \\
            h_s, & s > t_0,
        \end{cases}
    \end{equation}
    and set
    $$ 
    \tilde\tau(\omega,a,m,t) = \begin{cases}
        \tilde \tau_1(\omega) \wedge t_0, & t = t_0  \\ 
        t, & t > t_0
        \end{cases}.
    $$
    By construction, the triple $(\widetilde{M}, \widetilde{H},\tilde \tau)$ satisfies the required properties.
\end{proof}

\subsection{Integral characterization of strong $\beta$-biased order}

\begin{lemma}\label{lem:strict_order_simple_integrand}
    Let $H \geq 0$ be a predictable process satisfying $\mathbb{E}\left[\int_0^{t_\beta} H_s \, ds \right] < \infty$, and let $\tau$ be a stopping time such that $\tau < t_\beta$ a.s. Suppose there exist $0 <\hat t_0 < t_\beta$,
    a random variable $\bar{U} \in L^0(\mathcal{F}_0)$, and a deterministic function $\bar{H}$ such that $H_s \mathbf{1}_{s \leq \tau} = \bar{H}(s, \bar{U}) \mathbf{1}_{s \leq \tau_1}$ for all $s \in [0, \hat t_0]$. Then, 
    $$
    \delta_0 \sbbc \law\left(\int_0^\tau H_s \, dM_s\right).
    $$
\end{lemma}

\begin{proof}
    From Lemma 3.2 in \cite{Jacod75} it follows that there is $ \eta_0 \in L^0(\F_0)$ such that $\tau \wedge \tau_1 =  \eta_0 \wedge \tau_1$ a.s. Without loss of generality, we may assume that $\eta_0$ and $\bar U$ are constants. Indeed, if this is not the case, we condition on $\{ \eta_0 = \eta, \, \bar U = u\}$, and note that $\delta_0 \sbbc \law\left(\int_0^\tau H_s \, dM_s \big| \eta_0 = \eta, \; \bar U = u\right)$ implies $\delta_0 \sbbc \law\left(\int_0^\tau H_s \, dM_s\right)$. Since $\tau < t_\beta$ a.s., it follows that $\eta_0 < t_\beta$. We simplify the notation by writing $\bar{H}(s) := \bar{H}(s, \bar{U})$, since $\bar{U}$ is now assumed to be constant.

    Define $X_t := \int_0^t H_s \mathbf{1}_{s \leq \tau} \, dM_s$, $\nu_t := \law(X_t)$, and let $t_0 := \min(\hat{t}_0, \eta_0)$. 
    Since $\tau < t_\beta$ a.s., we have $\nu_{t_\beta} = \law\left(\int_0^{\tau} H_s \, dM_s\right)$. Since $X_{t_0} = \int_0^{t_0} \bar{H}(s) \mathbf{1}_{s \leq \tau_1} \, dM_s$ for a deterministic function $\bar{H}$, the distribution $\nu_{t_0}$ is atomic $e^{-t_0}$-biased.
    If $\nu_{t_0}$ is not simple, then $\tm{\nu_{t_0}} > e^{-t_0}$ and hence, by Lemma~\ref{lem:atomic_not_simple_implies_stronglybiased}, $\nu_{t_0}$ is strongly $e^{-t_0}$-biased. In this case, since $\nu_{t_0} \prec_{e^{-(t_\beta - t_0)}} \nu_{t_\beta}$, it follows from Theorem \ref{thm:transitivity} that $\delta_0 \sbbc \nu_{t_\beta}$. 
    Otherwise, if $\nu_{t_0}$ is simple $e^{-t_0}$-biased, define $\gamma_x := \law(X_{\eta_0} \mid X_{t_0} = x)$, $\pi_x := \law(X_{t_\beta} \mid X_{\eta_0} = x)$, $x_{t_0} := S(\nu_{t_0})$, and $x_{\eta_0} := S(\nu_{\eta_0})$. The following cases are possible:

    \emph{Case 1: $S(\gamma_{x_{t_0}}) < x_{\eta_0}$ (see Figure \ref{fig:case1}).}\\
    In this case, Lemma \ref{lem:sufficient_conditions_strict_order} (1) implies that $\delta_0 \prec_{s e^{-\eta_0}} \nu_{\eta_0}$, and by Theorem \ref{thm:transitivity} we conclude that $\delta_0 \sbbc \nu_{t_\beta}$.
    
    \emph{Case 2: $S(\gamma_{x_{t_0}}) = x_{\eta_0}$ (see Figure \ref{fig:case2}).}\\
    Here, $\nu_{\eta_0}$ is atomic $e^{-\eta_0}$-biased. If $\nu_{\eta_0}$ is not simple, i.e.\ $\tm{\nu_{\eta_0}} > e^{-\eta_0}$, then $\nu_{\eta_0}$ is strongly $e^{-\eta_0}$-biased, implying $\delta_0 \sbbc \nu_{t_\beta}$. Otherwise, since $X_{\eta_0} = x_{\eta_0}$ implies $\tau_1 \geq \eta_0$, and so $\tau = \eta_0$, then $\pi_{x_{\eta_0}} = \delta_{x_{\eta_0}}$, and Lemma \ref{lem:sufficient_conditions_strict_order} (2) again gives $\delta_0 \sbbc \nu_{t_\beta}$.
\end{proof}
\begin{figure}[ht]
    \centering
    \begin{subfigure}{0.45\textwidth}
        \includegraphics[width=\textwidth]{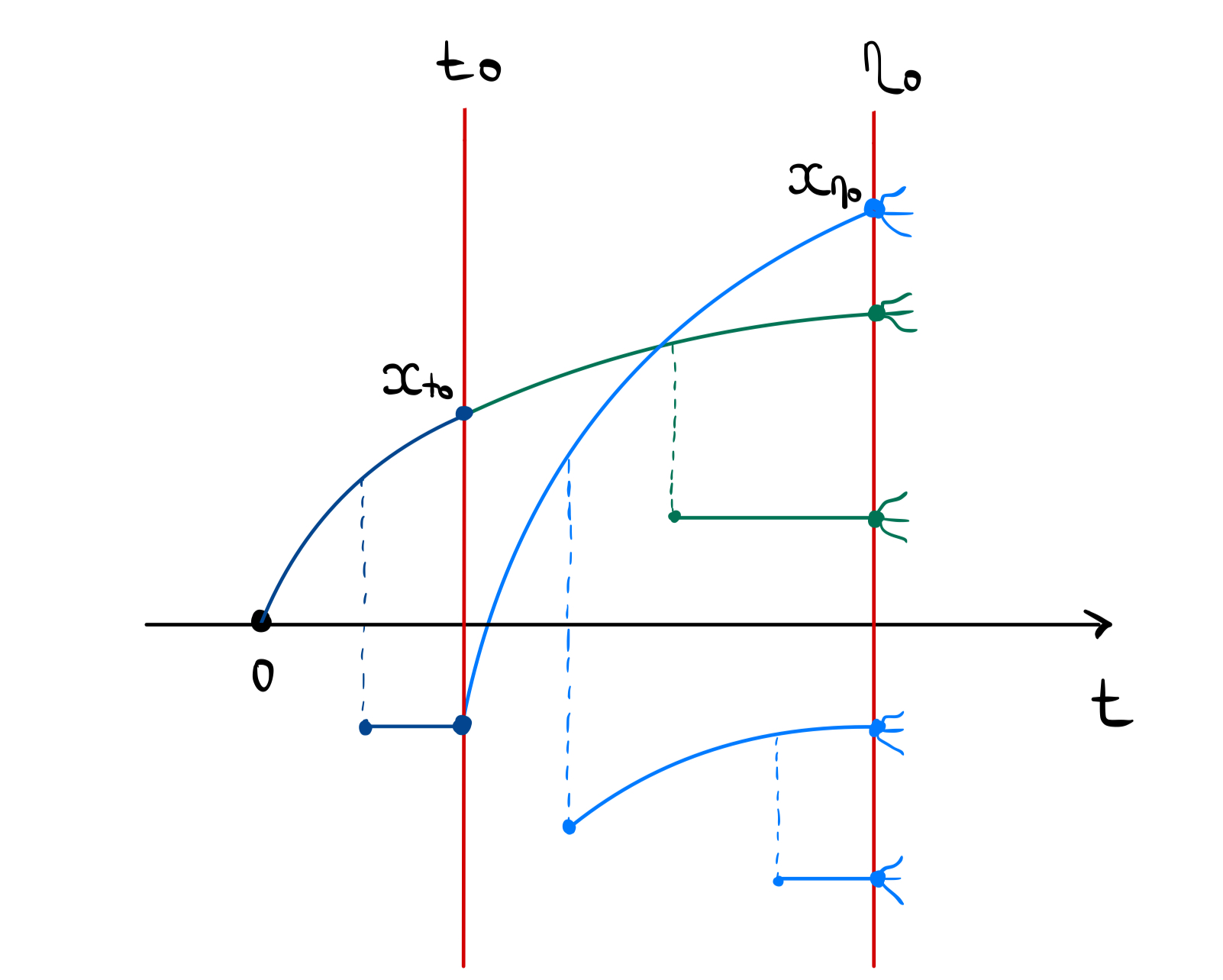}
        \caption{Case 1: $\max \supp(\gamma_{x_{t_0}}) < x_{\eta_0}$.}
        \label{fig:case1}
    \end{subfigure}
    \hfill
    \begin{subfigure}{0.45\textwidth}
        \includegraphics[width=\textwidth]{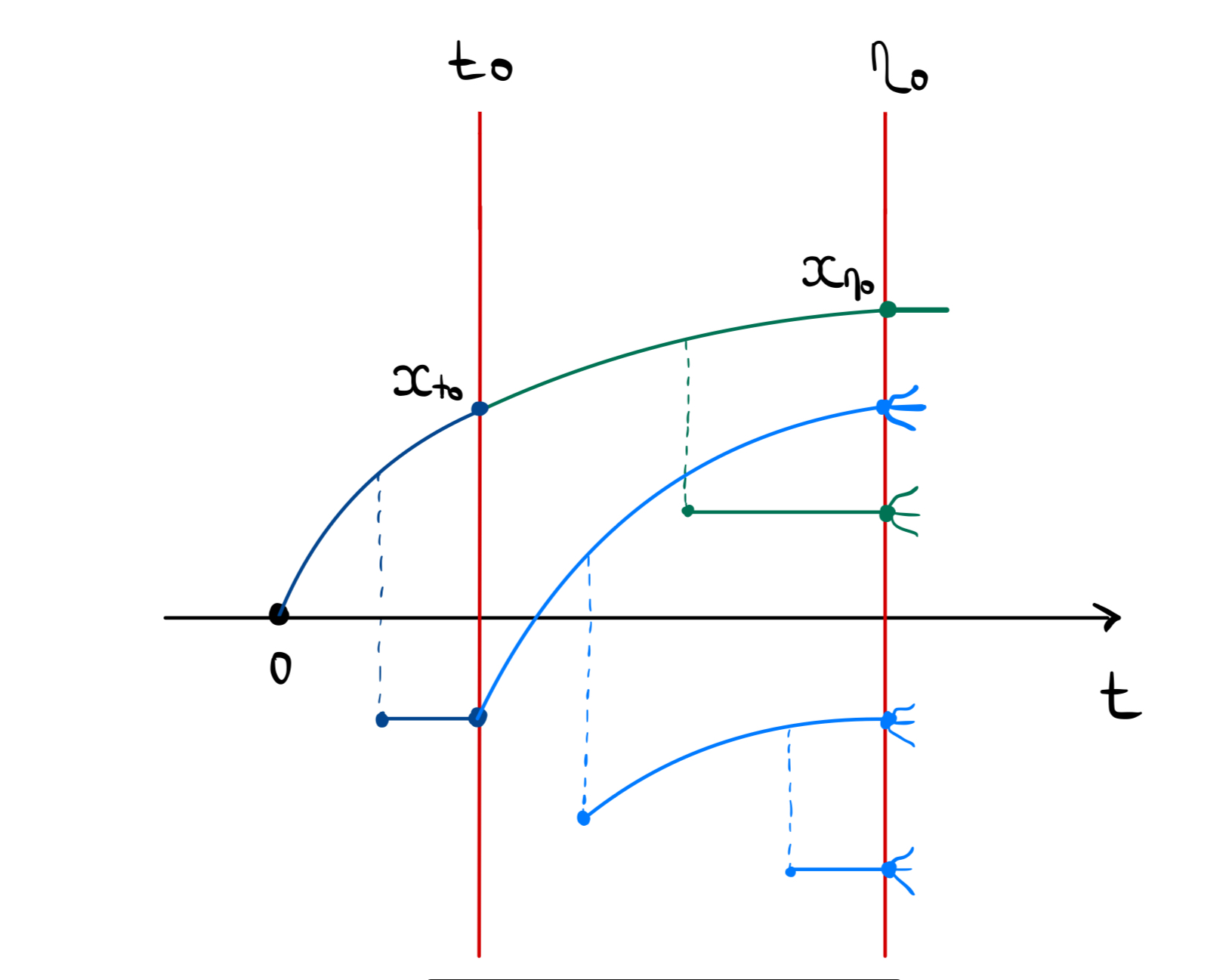}
        \caption{Case 2: $\max \supp(\gamma_{x_{t_0}}) = x_{\eta_0}$.}
        \label{fig:case2}
    \end{subfigure}
    \caption{The two possible configurations of $S(\gamma_{t_0})$ and $x_{\eta_0}$.
}
    \label{fig:cases}
\end{figure}

We are now in the position to prove our third main result (Theorem~\ref{thm:strict_order_integral_Representation_Intro}), which characterizes the strong $\beta$-biased order in terms of Poisson martingales.
    \begin{proof}[Proof of Theorem~\ref{thm:strict_order_integral_Representation_Intro}]
        First, we assume that $\mu \sbbc \nu$ and explicitly construct an integrand $H \geq 0$ and a stopping time $\tau$ such that \eqref{eq:strict_thm_order^Representation} is satisfied.
        Initially, we consider the case where $\mu = \delta_x$. Since $\delta_x \prec_{s\beta} \nu$, we can represent $\nu$ as a mixture $\nu = \int \nu_u \, \lambda(du)$, where a.e. $\nu_u$ is simple $\beta(u)$-biased around $x$ with  $\beta(u) := \nu_u([x, +\infty)) > \beta$. Since $u\mapsto\nu_u$ and $\nu_u \mapsto \beta(u)$ are measurable,  so is $u \mapsto \beta(u)$. By Corollary \ref{cor:integral_representation_simple_integrand}, there exists a deterministic function $\bar{H}(s,u)$ such that 
        $$
        x + \int_0^{t_{\beta(u)}} \bar{H}(s,u) \mathbf{1}_{s \leq \tau_1} \, dM_s \sim \nu_u,
        $$
        where $t_{\beta(u)} := -\log(\beta(u))$, and $\tau_1$ is the first jump time of the Poisson process $N$. Now, let $\hat{U} \in L^0(\F_0)$ be a random variable with $\hat{U} \sim \lambda$. Define $\tau := t_{\beta(\hat{U})} = -\log(\beta(\hat{U}))$ and set $H_s := \bar{H}(s,\hat{U}) \mathbf{1}_{s \leq \tau_1}$. Then, for any $z \in \mathbb{R}$,
        $$
        \P\left[x + \int_0^\tau H_s \, dM_s \leq z\right] 
        = \int \P\left[x + \int_0^\tau H_s \, dM_s \leq z \mid \hat{U} = u \right] \lambda(du) 
        $$
        $$
        = \int \P\left[x + \int_0^\tau \bar{H}(s, u) \, dM_s \leq z \right] \lambda(du) 
        = \nu((-\infty, z]),
        $$
        which implies that $x + \int_0^\tau H_s \, dM_s \sim \nu$.

        For the general case where $\mu \sbbc \nu$, consider a coupling $\pi \in \cplsbbc(\mu, \nu)$ and its $\mu$-disintegration $(\pi_x)_x$. 
         Since $\delta_x \prec_{s\beta} \pi_x$, we can measurably decompose $\pi_x$ into a mixture of simple $\beta(x,u)$-biased probabilities centered at $x$, i.e., $\pi_x = \int \pi_{x,u} \, \lambda(du)$, where each $\pi_{x,u}$ is simple $\beta(x,u)$-biased around $x$, with $\beta(x,u) := \pi_{x,u}([x, +\infty)) > \beta$, and $(x,u) \mapsto \pi_{x,u}$ is measurable. As above, we construct an integrand $H^x$ of the form $H_s^x = \hat H(s, \hat U, x)\mathbf{1}_{s \le \tau_1}$, where $\hat H$ is deterministic, and a stopping time $\tau_x < t_\beta$ of the form $\tau_x = -\log(\beta(x, \hat U))$ for some $\hat U \in L^0(\F_0)$, such that $x + \int_0^{\tau_x} H_s^x \, dM_s \sim \pi_x$ for $\mu$-a.e. $x$. By Remark~\ref{rem:meausurability_of_decomposition}, the map $x \mapsto H^x$ can be taken to be measurable. Next, take $X_0 \sim \mu$, independent of $\hat U$, and set $H_s := \hat H(s, \hat U, X_0)\mathbf{1}_{s \le \tau_1}$ and $\tau := \tau_{X_0} = -\log(\beta(X_0, \hat U))$. Then, by construction, we have $ X_0 + \int_0^\tau H_s \, dM_s \sim \nu. $

        For the reverse direction, assume that there exist a stopping time $\tau < t_\beta$ a.s. and a martingale $X$ such that $X_0 \sim \mu$ and $X_\tau \sim \nu$, and $X_\tau = X_0 + \int_0^\tau H_s \, dM_s$ for some predictable process $H \geq 0$ with $\E\left[\int_0^{t_\beta} H_s \, ds \right] < \infty$. Our goal is to show that this implies $\mu \sbbc \nu$. 
First, we consider the case where $X_0 = 0$, so our goal is to show that $\delta_0 \prec_{s\beta} \nu$. Denote $X_t := \int_0^t H_s \mathbf{1}_{s\leq \tau} , dM_s$, so that $X_{t_\beta} = X_\tau$.  By Lemma 3.2 in \cite{Jacod75}, there exists $\eta_0 \in L^0(\F_0)$ such that $\tau \wedge \tau_1 = \eta_0 \wedge \tau_1$ a.s. Without loss of generality, we may assume that $\eta_0$ is a constant. If this is not the case, we condition on ${\eta_0 = \eta}$ and note that $\delta_0 \sbbc \law\left(\int_0^\tau H_s , dM_s \mid \eta_0 = \eta\right)$ implies $\delta_0 \sbbc \nu$. Since $\tau < t_\beta$ a.s., it follows that $\eta_0 < t_\beta$.
        Now, fix $0<t_0 \leq \eta_0$. By Proposition \ref{prop:extension_of_integrand}, there exists an extension $(\widetilde{\Omega}, \widetilde{\F}, \widetilde{\P})$ of the probability space, equipped with a filtration $\widetilde{\mathbb{F}}=(\widetilde{\F}_s)_{s \in [0, T]}$, on which we can define processes $(\widetilde{H}_s)_{s \in [0, T]}$ and $(\widetilde{M}_s)_{s \in [0, T]}$, and a stopping time $\widetilde{\tau}$, such that:
        \begin{enumerate}
        \item $\widetilde{M}$ is a $(\widetilde{\P},\widetilde{\F})$-negative compensated Poisson process,
        \item $\widetilde{H} \geq 0$ is $\widetilde{\mathbb{F}}$-predictable 
        and satisfies $\mathbb{E}_{\widetilde{\P}}\left[\int_0^T \widetilde{H}_s \, ds\right] < \infty$,
        \item For $s \in [0, t_0]$, it holds that $\widetilde{H}_s \mathbf{1}_{s \leq \widetilde{\tau}} = \bar{H}(s, \bar{U}) \mathbf{1}_{s \leq \widetilde{\tau}_1}$ for some deterministic function $\bar{H}$, $\bar{U} \in L^0(\mathcal{F}_0)$, and where $\widetilde{\tau}_1$ is the first jump time of $\widetilde{M}$,
        \item For $Y_t := \int_0^t \widetilde{H}_s \mathbf{1}_{s \leq \widetilde{\tau}}\, d\widetilde{M}_s$, it holds that
    $Y_{t_0} \overset{d}{=} X_{t_0}$, and for $\law(X_{t_0})$-a.e. $x$,
    \begin{align*}
    \law_{\, \P}\!\big((H_s, M_s - M_{t_0})_{s \in [t_0,T]},\; \tau \vee t_0 \mid X_{t_0}=x\big) \\
        = \law_{\, \widetilde{\P}}\!\big((\widetilde{H}_s, \widetilde{M}_s - \widetilde{M}_{t_0})_{s \in [t_0,T]},\; 
           \widetilde{\tau} \vee t_0 \mid Y_{t_0}=x\big).
        \end{align*}
    \end{enumerate}
        Then, by construction, we have $\law_{\P}(X_\tau) = \law_{\widetilde \P}(Y_{\widetilde{\tau}})$. Applying Lemma~\ref{lem:strict_order_simple_integrand} to $(\widetilde H, \widetilde M, \widetilde \tau)$, we conclude that $\delta_0 \sbbc \law_{\widetilde \P}(Y_{\widetilde{\tau}})$, and therefore $\delta_0 \sbbc \law_{\P}(X_{\tau})$.

        In the case of a generic initial distribution $X_0 \sim \mu$, we condition on $X_0 = x$ and, as shown in the previous step, we derive $\delta_x \sbbc \law\left(x + \int_0^\tau H_s , dM_s \mid X_0 = x\right) =: \hat{\pi}_x$. Therefore, using $\law(X_\tau) = \int \hat{\pi}_x \, \mu(dx)$, we conclude that $\mu \sbbc \nu$.
\end{proof}

\noindent
{\bf Acknowledgment:}
 This research was funded in part by the Austrian Science Fund (FWF) [doi: 10.55776/P34743 and 10.55776/P35197]. For open access purposes, the author has applied a CC BY public copyright license to any author accepted manuscript version arising from this submission.

\bibliography{joint_biblio}

\begin{thebibliography}{19}
\providecommand{\natexlab}[1]{#1}
\providecommand{\url}[1]{\texttt{#1}}
\expandafter\ifx\csname urlstyle\endcsname\relax
  \providecommand{\doi}[1]{doi: #1}\else
  \providecommand{\doi}{doi: \begingroup \urlstyle{rm}\Url}\fi

\bibitem[Acciaio et~al.(2016)Acciaio, Beiglb{\"o}ck, Penkner, and
  Schachermayer]{AcBePeSc13}
B.~Acciaio, M.~Beiglb{\"o}ck, F.~Penkner, and W.~Schachermayer.
\newblock A model-free version of the fundamental theorem of asset pricing and
  the super-replication theorem.
\newblock \emph{Math. Finance}, 26\penalty0 (2):\penalty0 233--251, 2016.
\newblock ISSN 0960-1627.
\newblock \doi{10.1111/mafi.12060}.
\newblock URL \url{http://dx.doi.org/10.1111/mafi.12060}.

\bibitem[{A}cciaio et~al.(2025){A}cciaio, {B}eiglb{\"o}ck, {K}olosov, and
  {P}ammer]{AcBeKoPa24}
B.~{A}cciaio, M.~{B}eiglb{\"o}ck, E.~{K}olosov, and G.~{P}ammer.
\newblock On the arbitrage-free compatibility of {A}merican option prices.
\newblock \emph{In preparation}, 2025.

\bibitem[Acciaio et~al.(2025)Acciaio, Marini, and Pammer]{AcMaPa23}
B.~Acciaio, A.~Marini, and G.~Pammer.
\newblock Calibration of the bass local volatility model.
\newblock \emph{SIAM Journal on Financial Mathematics}, 2025.
\newblock \doi{10.1137/23M1622660}.
\newblock URL \url{https://doi.org/10.1137/23M1622660}.

\bibitem[{Backhoff-Veraguas} et~al.(2019){Backhoff-Veraguas}, Beiglb{\"o}ck,
  and Pammer]{BaBePa18}
J.~{Backhoff-Veraguas}, M.~Beiglb{\"o}ck, and G.~Pammer.
\newblock Existence, duality, and cyclical monotonicity for weak transport
  costs.
\newblock \emph{Calculus of Variations and Partial Differential Equations},
  58\penalty0 (6):\penalty0 1--28, 2019.
\newblock \doi{doi: 10.1137/18M1196479}.

\bibitem[Backhoff-Veraguas et~al.(2020)Backhoff-Veraguas, Beiglböck, Huesmann,
  and Källblad]{BaBeHuKa20}
J.~Backhoff-Veraguas, M.~Beiglböck, M.~Huesmann, and S.~Källblad.
\newblock Martingale benamou--brenier: A probabilistic perspective.
\newblock \emph{The Annals of Probability}, 48\penalty0 (5):\penalty0
  2258--2289, 2020.
\newblock \doi{10.1214/20-AOP1422}.
\newblock URL \url{https://doi.org/10.1214/20-AOP1422}.

\bibitem[{Backhoff-Veraguas} et~al.(2023){Backhoff-Veraguas}, {Beiglb\"ock},
  {Schachermayer}, and {Tschiderer}]{BaBeScTs23}
J.~{Backhoff-Veraguas}, M.~{Beiglb\"ock}, W.~{Schachermayer}, and
  B.~{Tschiderer}.
\newblock The structure of martingale {B}enamou--{B}renier in multiple
  dimensions.
\newblock \emph{ArXiv e-prints}, 2306.11019, 2023.

\bibitem[Beiglb{\"o}ck and Juillet(2016)]{BeJu16}
M.~Beiglb{\"o}ck and N.~Juillet.
\newblock On a problem of optimal transport under marginal martingale
  constraints.
\newblock \emph{Ann. Probab.}, 44\penalty0 (1):\penalty0 42--106, 2016.
\newblock ISSN 0091-1798.
\newblock \doi{10.1214/14-AOP966}.
\newblock URL \url{http://dx.doi.org/10.1214/14-AOP966}.

\bibitem[Beiglböck et~al.(2017)Beiglböck, Cox, and Huesmann]{BeCoHu14}
M.~Beiglböck, A.~M.~G. Cox, and M.~Huesmann.
\newblock Optimal transport and skorokhod embedding.
\newblock \emph{Inventiones Mathematicae}, 208\penalty0 (2):\penalty0 327--400,
  2017.
\newblock \doi{10.1007/s00222-016-0692-2}.
\newblock URL \url{https://doi.org/10.1007/s00222-016-0692-2}.

\bibitem[Breeden and Litzenberger(1978)]{BrLi78}
D.~T. Breeden and R.~H. Litzenberger.
\newblock Prices of state-contingent claims implicit in option prices.
\newblock \emph{The Journal of Business}, 51\penalty0 (4):\penalty0 621--51,
  1978.
\newblock URL
  \url{https://EconPapers.repec.org/RePEc:ucp:jnlbus:v:51:y:1978:i:4:p:621-51}.

\bibitem[Conze and Henry-Labordère(2021)]{CoHe21}
A.~Conze and P.~Henry-Labordère.
\newblock A new fast local volatility model.
\newblock \emph{Risk}, 2021.
\newblock URL \url{https://www.risk.net/media/download/1079736/download}.

\bibitem[Guo et~al.(2019)Guo, Loeper, and Wang]{GuLoWa19}
I.~Guo, G.~Loeper, and S.~Wang.
\newblock Local volatility calibration by optimal transport.
\newblock In \emph{2017 MATRIX Annals}, pages 51--64. Springer, 2019.

\bibitem[Hirsch and Yor(2009)]{HiYo09}
F.~Hirsch and M.~Yor.
\newblock {A construction of processes with one-dimensional martingale
  marginals, associated with a L{\'e}vy process, via its L{\'e}vy sheet}.
\newblock \emph{Journal of mathematics of Kyoto University}, 49\penalty0
  (4):\penalty0 785--815, 2009.

\bibitem[Hobson(2011)]{Ho11}
D.~Hobson.
\newblock The skorokhod embedding problem and model-independent bounds for
  option prices.
\newblock In \emph{Paris–Princeton Lectures on Mathematical Finance 2010},
  volume 2003 of \emph{Lecture Notes in Mathematics}, pages 267--318. Springer,
  2011.
\newblock \doi{10.1007/978-3-642-14660-2_4}.
\newblock URL \url{https://doi.org/10.1007/978-3-642-14660-2_4}.

\bibitem[Hobson and Neuberger(2012)]{HoNe12}
D.~Hobson and A.~Neuberger.
\newblock Robust bounds for forward start options.
\newblock \emph{Math. Finance}, 22\penalty0 (1):\penalty0 31--56, 2012.
\newblock ISSN 0960-1627.

\bibitem[Jacod(1975)]{Jacod75}
J.~Jacod.
\newblock Multivariate point processes: predictable projection, radon-nikodym
  derivatives, representation of martingales.
\newblock \emph{Zeitschrift für Wahrscheinlichkeitstheorie und Verwandte
  Gebiete}, 31\penalty0 (3):\penalty0 235--253, 1975.
\newblock \doi{10.1007/BF00536010}.
\newblock URL \url{https://doi.org/10.1007/BF00536010}.

\bibitem[Obłój(2004)]{Ob04}
J.~Obłój.
\newblock The skorokhod embedding problem and its offspring.
\newblock \emph{Probability Surveys}, 1:\penalty0 321--392, 2004.
\newblock \doi{10.1214/154957804100000060}.
\newblock URL
  \url{https://projecteuclid.org/journals/probability-surveys/volume-1/issue-none/The-Skorokhod-embedding-problem-and-its-offspring/10.1214/154957804100000060}.

\bibitem[Skorohod(1961)]{Sk61}
A.~V. Skorohod.
\newblock \emph{Issledovaniya po teorii sluchainykh protsessov
  ({S}tokhasticheskie differentsialnye uravneniya i predelnye teoremy dlya
  protsessov {M}arkova)}.
\newblock Izdat. Kiev. Univ., Kiev, 1961.

\bibitem[Strassen(1965)]{St65}
V.~Strassen.
\newblock The existence of probability measures with given marginals.
\newblock \emph{Ann. Math. Statist.}, 36:\penalty0 423--439, 1965.
\newblock ISSN 0003-4851.

\bibitem[Wiesel and Zhang(2023)]{WiZh22}
J.~Wiesel and E.~Zhang.
\newblock An optimal transport-based characterization of convex order.
\newblock \emph{Dependence Modeling}, 11\penalty0 (1), 2023.
\newblock \doi{10.1515/demo-2023-0102}.
\newblock URL \url{https://doi.org/10.1515/demo-2023-0102}.

\end{thebibliography}

\end{document}